\documentclass[10pt,a4paper]{article}

\usepackage{type1cm}        
\usepackage{makeidx}         
\usepackage{graphicx}        
\usepackage{multicol}        
\usepackage[bottom]{footmisc}

\makeindex             

\newtheorem{remark}{Remark}

\newtheorem{definition}{Definition}
\newtheorem{theorem}{Theorem}
\newtheorem{proof}{Proof}
\newtheorem{lemma}{Lemma}

\usepackage{graphicx}

\usepackage{color}
\usepackage{a4,amsmath,amssymb,amscd,latexsym} 
\usepackage{bbm} 
\usepackage{epsfig,color} 
\usepackage{amsfonts}
\usepackage{mathrsfs}
\usepackage{overpic}
\usepackage{paralist}
\usepackage{graphicx}
\usepackage{trfsigns}
\usepackage{url}
\usepackage{stmaryrd}
\usepackage{pict2e} 
\usepackage{caption,subcaption}
\usepackage{algorithmic,algorithm}
\usepackage{placeins} 

\usepackage{tikz}
\usetikzlibrary{calc,trees,positioning,arrows,chains,shapes.geometric,%
	decorations.pathreplacing,decorations.pathmorphing,shapes,%
	matrix,shapes.symbols,backgrounds,shadows}

\usepackage{tikz-cd} 

\usepackage{nicefrac}

\newcommand{\R}{\mathbb{R}}
\newcommand{\PP}{\mathbb{P}}

\newcommand{\pP}{\mathbb{P}}
\newcommand{\bV}{\mathbb{V}}
\newcommand{\cU}{\mathcal{U}}
\newcommand{\EE}{\mathbb{E}}

\renewcommand{\d}{\textup{d}}
\newcommand{\dx}{\textup{d}x}
\newcommand{\ds}{\textup{d}s}

\newcommand{\dt}{\textup{d}t}

\DeclareMathOperator{\tr}{tr}

\newcommand{\footremember}[2]{%
    \footnote{#2}
    \newcounter{#1}
    \setcounter{#1}{\value{footnote}}%
}

\newcommand{\restr}[2]{#1\rvert_{#2}}

\usepackage{todonotes}

\begin{document}

\title{PDE-constrained shape optimization: towards product shape spaces and stochastic models\thanks{This work has been partly supported by the state of Hamburg within the Landesforschungsf\"orderung  under  project ``Simulation-Based Design Optimization of Dynamic Systems Under Uncertainties'' (SENSUS) with project number LFF-GK11, and by the German Academic Exchange Service (DAAD) within the program ``Research Grants-Doctoral Programmes in Germany, 2017/18."}}


\author{Caroline Geiersbach \and Estefania  Loayza-Romero \and  Kathrin Welker}


\author{Caroline Geiersbach\footremember{a}{Weierstrass Institute,
	\texttt{caroline.geiersbach@wias-berlin.de}} \and Estefania Loayza-Romero\footremember{b}{Chemnitz University of Technology, \texttt{estefania.loayza@math.tu-chemnitz.de}} \and Kathrin Welker\footremember{c}{Helmut-Schmidt-University / University of the Federal Armed Forces Hamburg, \texttt{welker@hsu-hh.de}}
}

\date{\today}

\maketitle


\abstract{
Shape optimization models with one or more shapes are considered in this chapter. Of particular interest for applications are problems in which where a so-called shape functional is constrained by a partial differential equation (PDE) describing the underlying physics. A connection can made between a classical view of shape optimization and the differential-geometric structure of shape spaces. To handle problems where a shape functional depends on multiple shapes, a theoretical framework is presented, whereby the optimization variable can be represented as a vector of shapes belonging to a product shape space. The multi-shape gradient and multi-shape derivative are defined, which allows for a rigorous justification of a steepest descent method with Armijo backtracking. As long as the shapes as subsets of a hold-all domain do not intersect, solving a single deformation equation is enough to provide descent directions with respect to each shape. Additionally, a framework for handling uncertainties arising from inputs or parameters in the PDE is presented. To handle potentially high-dimensional stochastic spaces, a stochastic gradient method is proposed. A model problem is constructed, demonstrating how uncertainty can be introduced into the problem and the objective can be transformed by use of the expectation. Finally, numerical experiments in the deterministic and stochastic case are devised, which demonstrate the effectiveness of the presented algorithms. 
}

%
%

\section{Introduction}

Shape optimization is concerned with problems in which an objective function is supposed to be minimized with respect to a shape, or a subset of $\R^{d}$.
One challenge in shape optimization is finding the correct model to describe the set of shapes; another is finding a way to handle the lack of vector structure of the shape space.
In principle, a finite dimensional optimization problem can be obtained for example by representing shapes as splines. However, this representation limits the admissible set of shapes, and the connection of shape calculus with infinite dimensional spaces \cite{Delfour-Zolesio-2001,SokoZol} leads to a more flexible approach. It was suggested to embed shape optimization problems in the framework of optimization on shape spaces \cite{Schulz,Welker2016}. One possible approach is to cast the sets of shapes in a Riemannian viewpoint, where each shape is a point on an abstract manifold equipped with a notion of distances between shapes; see, e.g., \cite{MichorMumford2,MichorMumford1}. From a theoretical and computational point of view, it is attractive to optimize in Riemannian shape manifolds because algorithmic ideas from \cite{Absil} can be combined with approaches from differential geometry. 
Here, the Riemannian shape gradient 
can be used to solve such shape optimization problems using the gradient descent method.
In the past, major effort in shape calculus has been devoted towards expressions for shape derivatives in the so-called Hadamard form, which are integrals over the surface (cf.~\cite{Delfour-Zolesio-2001,SokoZol}). 
During the calculation of these expressions, volume shape derivative terms arise as an intermediate result. 
In general, additional regularity assumptions are necessary in order to transform the volume forms into surface forms.
Besides saving analytical effort, this makes volume expressions preferable to Hadamard forms.
In this chapter,  the Steklov--Poincar\'{e} metric is considered, which allows to use the volume formulations (cf.~\cite{SchulzSiebenbornWelker2015:2}). The reader is referred to \cite{HardestyKouriLindsayRidzalStevensViertel:2020:1,HiptmairPaganiniSargheini:2015:1}  for a comparison on the volume and boundary formulations with respect to their order of convergence in a finite element setting.

In applications, often more than one shape needs to be considered, e.g., in electrical impedance tomography, where the material distribution of electrical properties such as electric conductivity and permittivity inside the body is examined \cite{Cheney1999,Kwon2002,laurain2016distributed}, 
and the optimization of biological cell composites in the human skin \cite{SiebenbornNaegel,Siebenborn2017}.  
If a shape is seen as a point on an abstract manifold, it is natural to view a collection of shapes as a vector of points. Using this perspective, a shape optimization problem can be formulated over multiple shapes. This novel, multi-shape optimization problem is developed in this chapter. 

A second area of focus in this chapter is in the development of stochastic models for multi-shape optimization problems.
There is an increasing effort to incorporate uncertainty into shape optimization models; see, for instance \cite{Dambrine2015, dambrine2019incorporating, Hiptmair2018, Liu2017, Martinez-Frutos2016}. 
Many relevant problems contain additional constraints in the form of a PDE, which describe the physical laws that the shape should obey. Often, material coefficients and external inputs might not be known exactly, but rather be randomly distributed according to a probability distribution obtained empirically. In this case, one might still wish to optimize over a set of these possibilities to obtain a more robust shape. When the number of possible scenarios in the probability space is small, then the optimization problem can be solved over the entire set of scenarios. This approach is not relevant for most applications, as it becomes intractable if the random field has more than a few scenarios. For problems with PDEs containing uncertain inputs or parameters, either the stochastic space is discretized, or sampling methods are used. If the stochastic space is discretized, one typically relies on a finite-dimension assumption, where a truncated expansion is used as an approximation of the infinite-dimensional random field. Numerical methods include stochastic Galerkin method \cite{Babuska2004} and sparse-tensor discretization \cite{Schwab2011}. Sample-based approaches involve taking random or carefully chosen realizations of the input parameters; this includes Monte Carlo or quasi Monte Carlo methods and stochastic collocation \cite{Babuska2007}. In the stochastic approximation approach, dating back to a chapter by Robbins and Monro \cite{Robbins1951}, one uses a stochastic gradient in place of a gradient to iteratively minimize the expected value over a random function. Recently, stochastic approximation was proposed to solve problems formulated over a shape space that contain uncertainties \cite{GeiersbachLoayzaWelker}. A novel stochastic gradient method was formulated over infinite-dimensional shape spaces and convergence of the method was proven. The work was informed by its demonstrated success in the context of PDE-constrained optimization under uncertainty \cite{Geiersbach2019, Haber2012, Martin2018, Geiersbach2020b, Geiersbach2021}. 

The chapter is structured as follows. Section~\ref{sec:OptProductManifolds} is concerned with deterministic shape optimization. First, in subsection~\ref{section:DetOpt}, it is summarized how the theory of deterministic PDE-constrained shape optimization problems can be connected with the differential-geometric structure of the space of smooth shapes. The novel contribution of this chapter is in subsection~\ref{section:product-manifold}, which concentrates on more than one shape to be optimized in the optimization model. 
A framework is introduced to justify a mesh deformation method using a Steklov--Poincar\'e metric defined on a product manifold. This novel framework is further developed in section~\ref{section:stochastic} in the context of shape optimization under uncertainty. The stochastic gradient method is revisited in the context of problems depending on mutiple shapes. Numerical experiments demonstrating the effectiveness of the deterministic and stochastic methods are shown in section~\ref{sec:numerical_experiments}. Finally, closing remarks are shared in section~\ref{sec:conclusion}.

\section{Optimization over product shape manifolds}
\label{sec:OptProductManifolds}

This chapter is concerned with class of optimization problems, where the optimization variable is a vector $u=(u_1, \dots, u_N)$ of non-intersecting shapes  contained in a bounded domain $D \subset \R^{d}$ as shown in figure~\ref{fig_domain} for $d=2$ and $N=5$. This domain will sometimes be called the hold-all domain, and its boundary is denoted by $\partial D$.
The outer normal vector field $\operatorname{n}$ on a shape $u\in\mathcal{U}^N$ is defined by $\operatorname{n}=(\operatorname{n}_1,\dots, \operatorname{n}_N)$, where $\operatorname{n}_i$ denotes the unit outward normal vector field to $u_i$ for $i=1,\dots, N$.

Next, the shape space concept considered in this chapter needs to be clarified. Shapes space definitions have been extensively studied in recent decades. 
Already in 1984, \cite{Kendall} introduced the notion of a shape space. Here, a shape space is just modeled as a linear (vector) space, which in the simplest case is made up of vectors of landmark positions.
However, there is a large number of different shape concepts, e.g., plane curves \cite{MichorMumford}, surfaces in higher dimensions \cite{BauerHarmsMichor,MichorMumford2}, boundary contours of objects \cite{FuchsJuettlerScherzerYang,LingJacobs,RumpfWirth2}, multiphase objects \cite{WirthRumpf}, characteristic functions of measurable sets \cite{Zolesio},  morphologies of images \cite{DroskeRumpf}, and planar triangular meshes \cite{HerzogLoayzaRomero:2020:1}.
In a lot of processes in engineering, medical imaging, and science, there is a great interest to equip the space of all shapes with a significant metric to distinguish between different shape geometries.
In the simplest shape space case (landmark vectors), the distances between shapes can be measured by the Euclidean distance, but in general, the study of shapes and their similarities is a central problem.
In contrast to a parametric optimization problem, which can be obtained, e.g., by representing shapes as splines, the connection of shape calculus with infinite dimensional spaces \cite{Delfour-Zolesio-2001,ItoKunisch,SokoZol} leads to a more flexible approach. 
As already mentioned, solving PDE-constrained shape optimization problems under a differential geometric paradigm has various advantages  \cite{Schulz2015a}, one of them being the opportunity to obtain a natural measure of similarity of shapes through the Riemannian metric. Moreover, depending on the metric defined over a manifold, different goals can be achieved. This chapter focuses on the Steklov--Poincare metric \cite{SchulzSiebenbornWelker2015:2} because of its direct relation to the finite element method.

In view of using the Steklov--Poincar\'{e} metric, this chapter concentrates on shape spaces as Riemannian manifolds. Thus, it is assumed $u_i\in \mathcal{U}_i$ for all $i=1,\dots,N$ for Riemannian manifolds $(\mathcal{U}_i, G^i)$, i.e., $ u$ is an element of the \emph{product shape space} $  \mathcal{U}^N \colon= \cU_1 \times \dots \times \cU_N= \prod_{i=1}^{N}\mathcal{U}_i$.
If there is only one shape, the notation $ \mathcal{U}$ instead of $ \mathcal{U}^1$ is used.
Since a Riemannian metric $G^i$ varies with the point of evaluation, it will be denoted $G_{p}^i(\cdot,\cdot)\colon T_{p} \mathcal{U}_i \times T_{p} \cU_i \to \R$, to highlight its dependence on the point $p$.  
Hereby, the \emph{tangent space} at a point $p \in \mathcal{U}_i$ is defined in its geometric version as 
$$T_p\mathcal{U}_i=\{ c\colon \R \rightarrow \mathcal{U}_i: c\text{ differentiable}, c(0) =p\}/\sim,$$
where the equivalence relation for two differentiable curves  $c,\tilde{c}\colon\R \rightarrow \mathcal{U}_i$  with $c(0) = \tilde{c}(0) =p$ is defined as follows:
\[
c \sim \tilde{c} \Leftrightarrow \tfrac{\d}{\dt}\phi_{\alpha}(c(t))\vert_{t=0}  =\tfrac{\d}{\dt} \phi_{\alpha}(\tilde{c}(t))\vert_{t=0}\,\forall\,\alpha\text{ with }u \in U_\alpha,
\]
where $\{(U_\alpha, \phi_\alpha)\}_\alpha$ is the atlas of $\mathcal{U}_i$.


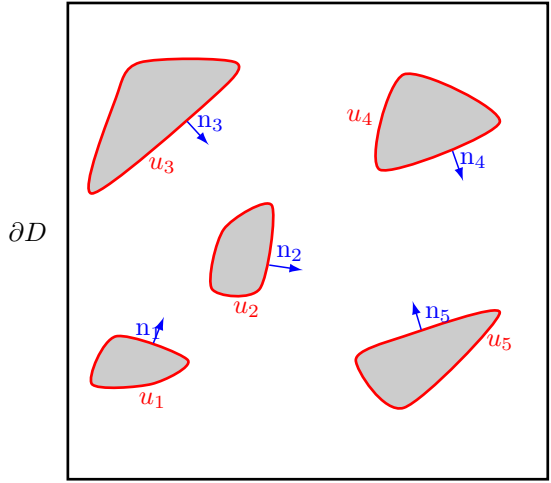
\begin{figure}
	\begin{center}
		\begin{overpic}[width=.5\textwidth]{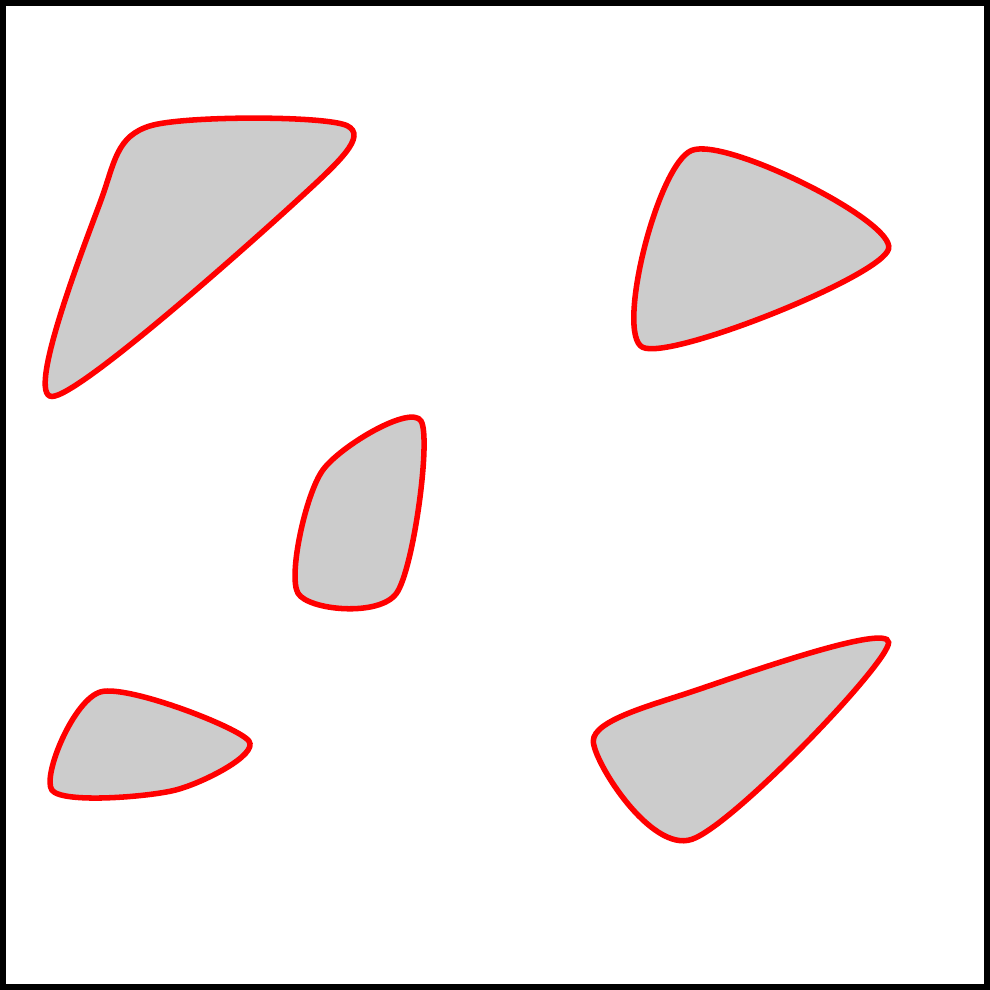}
			\put (-12,50){$\partial D$}
			\put (15,16){\textcolor{red}{$u_1$}}
			\put (18, 28.7){\color{blue}\linethickness{0.2mm}\vector(1,2.5){2.2}}
			\put (14.3, 30.3){\textcolor{blue}{$\textup{n}_1$}}
			\put (34.5,35){\textcolor{red}{$u_2$}}	
			\put (42, 45){\color{blue}\linethickness{0.2mm}\vector(1,-0.15){7}}
			\put (43.5, 46.6){\textcolor{blue}{$\textup{n}_2$}}
			\put (17,65){\textcolor{red}{$u_3$}}
			\put (25, 75){\color{blue}\linethickness{0.2mm}\vector(1,-1.1){4.5}}
			\put (27.1, 73.8){\textcolor{blue}{$\textup{n}_3$}}
			\put (58,75){\textcolor{red}{$u_4$}}
			\put (80, 68.8){\color{blue}\linethickness{0.2mm}\vector(0.7,-2){2.2}}
			\put (81.5, 66){\textcolor{blue}{$\textup{n}_4$}}
			\put (87,28){\textcolor{red}{$u_5$}}
			\put (73.5, 31.5){\color{blue}\linethickness{0.2mm}\vector(-0.6,2){1.8}}
			\put (74.3, 34){\textcolor{blue}{$\textup{n}_5$}}
		\end{overpic}
	\end{center}
	\caption{Illustration of the domain $D$ in $\R^2$ for $N=5$.}
	\label{fig_domain}
\end{figure}

A main focus in shape optimization is in the investigation of 
shape functionals.
A shape functional on $\mathcal{U}^N$ is given by a function 
$$j\colon \mathcal{U}^N \to \mathbb{R}\text{, } u\mapsto j(u).$$
An \emph{unconstrained shape optimization problem} is given by
\begin{equation}
\label{minproblem}
\min_{u\in \mathcal{U}^N } j(u).
\end{equation}
Often, shape optimization problems are constrained by equations, e.g., equations involving an unknown function of two or more variables and at least one partial derivative of this function. The objective may depend on not only the shapes $u$ but also the  \emph{state variable} $y$, where the state variable is the solution of the underlying constraint. In other words, one has a shape functional of the form $\hat{j}\colon \mathcal{U}^N \times \mathcal{Y} \rightarrow \R$ and an operator $e\colon \mathcal{U}^N \times \mathcal{Y} \rightarrow \mathcal{W}$, where $\mathcal{Y}$ and $\mathcal{W}$ are Banach spaces. One therefore has a \emph{constrained shape optimization problem} of the form 
\begin{equation}
\begin{aligned}
&\min_{(u,y)\in \mathcal{U}^N \times \mathcal{Y}} \hat{j}(u,y)\label{minproblem2}\\
&\hspace*{.7cm}\text{s.t. } \,\, 
e(u,y)=0.
\end{aligned}
\end{equation}
When $e$ in \eqref{minproblem2} represents a PDE,  the shape optimization problem is called \emph{PDE-constrained}. Formally, if the PDE has a (unique) solution given any choice of $u$, then the so-called \textit{control-to-state operator} $S\colon \mathcal{U}^N \rightarrow \mathcal{Y}$, $u \mapsto y$ is well-defined. With $j(u) := \hat{j}(u, S u)$ one obtains an unconstrained optimization problem of the form \eqref{minproblem}. This observation justifies the following work with \eqref{minproblem}, although later in the application section, a problem of the form \eqref{minproblem2} is presented.

Subsection \ref{section:DetOpt} concentrates on $N=1$ and summarizes how the theory of deterministic PDE-constrained shape optimization problems can be connected to the differential-geometric structure of shape spaces. Here, in view of obtaining efficient gradient based algorithms one focuses on the Steklov--Poincar\'{e} metric considered  in \cite{SchulzSiebenbornWelker2015:2}. Afterwards, subsection \ref{section:product-manifold} concentrates on $N>1$, which lead to product shape manifolds. It will be shown that it is possible to define a product metric and use this to justify the main result of this chapter, theorem \ref{thm:equivalence_variational_formulations}. It is rigorously argued that vector fields induced by the  shape derivative give descent directions with respect to each individual element of the shape space as well as the corresponding element of the product shape space.

\subsection{Optimization on shape spaces with Steklov--Poincar\'{e} \\ metric}
\label{section:DetOpt}

In this subsection, optimization with respect to one shape $u\in \mathcal{U}$ is discussed, i.e., $N=1$ is chosen. Additionally, the connection between Riemannian geometry on the space of smooth shapes and shape optimization is analyzed. Please note in the following, one shape is both an element of a manifold and a subset of $\R^{d}$.
In classical shape calculus, a shape is considered to be a subset of $\R^{d}$, only. However, this subsection explains that equipping a shape with additional structure provides theoretical advantages, enabling the use of concepts from differential geometry like the pushforward, exponential maps, etc.

\bigskip

\noindent
\textbf{Shape calculus.}
First,  notation and terminology of basic shape optimization concepts will be set up. 
For a detailed introduction into shape calculus, the reader is refereed to the monographs \cite{Delfour-Zolesio-2001,SokoZol}. 
The concept of shape derivatives is needed.  In order to define these derivatives, one concentrates on the shape $u$ as subset of $D\subset \R^{d}$ and considers a
family $\{F_t\}_{t\in[0,T]}$ of mappings $F_t\colon \overline{D}\to\mathbb{R}^d$ 
 such that $F_0=\operatorname{id}$, where $\overline{D}$ denotes the closure of $D$ and $T>0$.
This family transforms shapes $u$ into new \emph{perturbed shapes}
$$F_t(u) = \{  F_t(x)\colon x\in u\} . $$
Such a transformation can be described by the \emph{velocity method} or by the \emph{perturbation of identity}; cf.~\cite[pages 45 and 49]{SokoZol}. 
In the following, the perturbation of identity is considered. It is defined by $F_t^W(x):= x+tW(x)$, where $W\colon \overline{D} \rightarrow \R^{d}$ denotes a sufficiently smooth vector field.

\begin{definition}[Shape derivative] 
	Let $D\subset \mathbb{R}^{d}$ be open, $u\subset D$ and $k\in\mathbb{N}\cup \{\infty\}$.
The Eulerian derivative of a shape functional $j$ at $u$ in direction $W\in\mathcal{C}^k_0(D,\mathbb{R}^d)$ is defined by 
\begin{equation}
\label{eulerian}
dj(u)[W]:= \lim\limits_{t\to 0^+}\frac{j(F_t^W(u))-j(u)}{t}. 
\end{equation}
If for all directions $W\in\mathcal{C}^k_0(D,\mathbb{R}^d)$ the Eulerian derivative \eqref{eulerian} exists and the mapping 
\begin{equation*}
 \mathcal{C}^k_0(D,\mathbb{R}^d)\to \mathbb{R}, \ W\mapsto dj(u)[W]
\end{equation*}
is linear and continuous, the expression $dj(u)[W]$ is called the shape derivative of $j$ at $u$ in direction $W\in\mathcal{C}^k_0(D,\mathbb{R}^d)$. In this case, $j$ is called shape differentiable of class $\mathcal{C}^k$ at $u$.
	\label{def_shapeder}
\end{definition}

The proof of existence of shape derivatives can be done via different approaches like the Lagrangian~\cite{Sturm2013}, min-max~\cite{Delfour-Zolesio-2001}, chain rule~\cite{SokoZol}, rearrangement~\cite{Ito-Kunisch-Peichl} methods, among others. 
If the objective functional is given by a volume integral, under the assumptions of the Hadamard Structure Theorem (cf.~\cite[Theorem 2.27]{SokoZol}), the shape derivative can be expressed as an integral over the domain, the so-called \emph{volume} or \emph{weak formulation}, and  also as an integral over the boundary, the so-called \emph{surface} or \emph{strong formulation}.
Recent advances in PDE-constrained optimization on shape manifolds are based on the surface formulation, also called \emph{Hadamard-form}, as well as intrinsic shape metrics. Major effort in shape calculus has been devoted towards such surface expressions (cf.~\cite{Delfour-Zolesio-2001,SokoZol}), which are often very tedious to derive.
When one derives a shape derivative of an objective functional, which is given by an integral over the domain, one first gets the volume formulation. This volume form can be converted into its surface form by applying the integration by parts formula. In order to apply this formula, one needs a higher regularity of the state and adjoint of the underlying PDE. 
Recently, it has been shown that the weak formulation has numerical advantages, see, for instance, \cite{Berggren,Langer-2015,HipPag_2015,Paganini}. In \cite{HardestyKouriLindsayRidzalStevensViertel:2020:1, LaurainSturm2013}, practical advantages of volume shape formulations have also been demonstrated.

\bigskip

\noindent
\textbf{Shape calculus combined with differential geometric structure of shape manifolds.}
Solving shape optimization problems is made more difficult by the fact that the set of permissible shapes  generally does not allow a vector space structure, which is one of the main difficulties for the formulation of efficient optimization methods. In particular, without a vector space structure, there is no obvious distance measure, which is needed to establish convergence properties. In many practical applications, this difficulty is circumvented by characterizing the shapes of interest by finitely many parameters such that the parameters are elements of a vector space. Often, a priori parametrizations of the shapes of interest are used because of the resulting vector space framework matching standard optimization software. However, this limits the insight into the optimal shapes severely, because only shapes corresponding to the a priori parametrization can be reached. 
One possibility to avoid this limitation would be to focus on shape optimization in the setting of shape spaces. If one cannot work in vector spaces, shape spaces which allow a Riemannian structure like Riemannian manifolds are the next best option. 

Now, a shape $u \subset D$ is viewed also as an element of a Riemannian shape manifold $(\mathcal{U},G)$.  This means that the shape functional $J$ is defined on the manifold. Next, the derivative of a scalar field $j\colon \mathcal{U}\to\R$ needs to be defined.

\begin{definition}[Pushforward]
For each point $u\in\mathcal{U}$, the pushforward associated with $j\colon \mathcal{U}\to\R$ is given by the map
$$(j_\ast)_u\colon T_u\mathcal{U}\to \R, \, c\mapsto \frac{\d}{\d t} j(c(t)) \vert_{t=0}  =(j\circ c)'(0) .$$
\label{Def:Pushforward}
\end{definition}

\begin{remark}
In general, the pushforward is defined for a map $f$ between two differential manifolds $M$ and $N$. The definition depends on the used tangent space. In this setting, where tangent spaces are defined as equivalence classes of curves, the pushforward of $f\colon M\to N$ at a point $p\in M$ is generally given by a map between the tangent spaces, i.e.,
$ ( f_\ast)_p\colon T_pM\to T_{f(p)}N$ with $ (f_\ast)_p (c):=\frac{\d}{\d t}  f(c(t)) \vert_{t=0}  =(f\circ c)'(0) .$
\end{remark}

With the help of the pushforward, it is possible to define the Riemannian shape gradient.

\begin{definition}[Riemannian shape gradient]
	Let $(\mathcal{U},G)$ be a Riemannian manifold and $j\colon\mathcal{U}\to\R$.
A Riemannian shape gradient $\nabla j(u)\in T_u\mathcal U$ is defined by the relation
\begin{equation*}
\label{ShapeGradient}
(j_\ast)_u w= G_u(\nabla j(u),w)\quad \forall\, w\in T_u\mathcal{U}.
\end{equation*}
\end{definition}

Thanks to the definition of the Riemannian shape gradient, it is possible to formulate the gradient method on the Riemannian manifold $(\mathcal{U},G)$ (cf.~algorithm~\ref{Algo:Gradient}).
The Riemannian shape gradient with respect to $G$ is computed from \eqref{CompRiemShapeGradient}. The negative solution $-v^k$ is then used as descent direction for the objective functional $j$ in each iteration $k$. 
In order to update the shape iterates, the exponential map in algorithm~\ref{Algo:Gradient} is used; because the calculations of optimization methods on manifolds have to be performed in tangent spaces, points from a tangent space have to be mapped to the manifold in order to define the next iterate. Figure \ref{fig:exponential} illustrates this situation. With \eqref{Update_gS}
the $(k+1)$-th shape iterate $u^{k+1}$ is calculated, 
where $\exp_{u^k}\colon T_{u^k}\mathcal{U}\to \mathcal{U},\,z\mapsto \exp_{u^k}(z)$ denotes the exponential map; this defines a local diffeomorphism between the tangent space $T_{u^k}\mathcal{U}$ and the manifold $\mathcal{U}$ by following the locally uniquely defined geodesic starting in the $k$-{th} shape iterate $u^k\in\mathcal{U}$ in the direction $-v^k \in T_{u^k}\mathcal{U}$. 
In algorithm \ref{Algo:Gradient}, an Armijo backtracking line search technique is used to calculate the step-size $t^k$ in each iteration. Here, the norm introduced by the metric under consideration is needed, $\|\cdot\|_{G}:=\sqrt{G(\cdot,\cdot)}$. 

\begin{algorithm}
	\caption{Steepest descent method on $(\mathcal{U},G)$ with Armijo backtracking line search}
	\label{Algo:Gradient}
	\begin{algorithmic}[0]
		\STATE\textbf{Require:} Objective function $j$ on $(\mathcal{U},G)$
		\vspace{.1cm}
		\STATE \textbf{Input:} Initial shape $u^0\in \mathcal{U}$ \hspace*{15cm}
		
		\hspace*{1.1cm}constants $\hat{\alpha}>0$ and $	{\displaystyle \sigma, \rho \in (0,1)}$ for Armijo backtracking strategy
		\vspace{.3cm}
		\STATE \textbf{for} $k=0,1,\dots$ \textbf{do}
		\vspace{.1cm}
		\STATE [1]  Compute the Riemannian shape gradient $v^k\in T_{u^k}\mathcal{U}$ with respect to $G$ by solving
		\begin{equation}
		\label{CompRiemShapeGradient}
			(j_\ast)_{u^k}w=G_{u^k}(v^k,w)\quad \forall\, w\in T_{u^k}\mathcal{U}.
		\end{equation}
		\STATE [2] Compute Armijo backtracking step-size: 
		\vspace{.1cm}
		\STATE \hspace*{1cm} Set $\alpha := \hat{\alpha}$.
		\STATE \hspace*{1cm} \textbf{while} $ j(\operatorname{exp}_{u^k}(-\alpha v^k)) > j(u^k)-\sigma\alpha \left\|v^k\right\|^2_{G}$ 
		\STATE \hspace*{1cm} Set $ \alpha :=\rho \alpha $.
		\STATE \hspace*{1cm} \textbf{end while}
		\STATE \hspace*{1cm} Set $t^k:=\alpha$.
		\vspace{.1cm}
		\STATE [3] Set 
		\begin{equation}
		\label{Update_gS}
		u^{k+1}:= \operatorname{exp}_{u^k}(-t^k v^k).
		\end{equation}
		\STATE \textbf{end for}
		\vspace{.3cm}
	\end{algorithmic}
\end{algorithm}

\bigskip
\noindent
\textbf{Optimization on the space of smooth shapes.}
This chapter focuses on the manifold of $d$-dimensional smooth shapes. 
The \emph{set of all $(d-1)$-dimensional smooth shapes}  is considered in \cite{MichorMumford2} and can be characterized by
\begin{equation*}
\label{B_e_2dim}
B_e= B_e(S^{d-1},\R^{d}):= \mathrm{Emb}(S^{d-1},\R^{d})/\mathrm{Diff}(S^{d-1}).
\end{equation*}
Here, $\mathrm{Emb}(S^{d-1},\R^{d})$ denotes the set of all embeddings from the unit circle $S^{d-1}$ into $\R^{d}$, and $\mathrm{Diff}(S^{d-1})$ is the set of all diffeomorphisms from $S^{d-1}$ into itself.  
In \cite{KrieglMichor}, it is verified that the shape space $B_e$ is a smooth manifold. 
The tangent space is isomorphic to the set of all smooth normal vector fields along $c$, i.e.,
\begin{equation*}
\label{isomorphismTcBe}
T_uB_e(S^{d-1},\mathbb{R}^{d})\cong\left\{h\colon h=\alpha \operatorname{n},\, \alpha\in \mathcal{C}^\infty(S^{d-1})\right\},
\end{equation*}
where $\operatorname{n}$ denotes the outer unit normal field to the shape $u$. 
Next, the connection of shape derivatives with the geometric structure of $B_e$ is addressed. This combination results in efficient optimization techniques on $B_e$.

In view of obtaining gradient-based optimization approaches, the gradient needs to be specified. The gradient will be characterized by the chosen Riemannian metric on $B_e$. 
Several Riemannian metrics on this shape space are examined, e.g., \cite{BauerHarmsMichor,MichorMumford2,MichorMumford}. All these metrics arise from the $L^2$-metric by putting weights, derivatives or both in it. 
 In this manner, one gets three groups of metrics: the \emph{almost local metrics} which arise by putting weights in the $L^2$-metric (cf.~\cite{BauerHarmsMichor_SobolevII,MichorMumford}), the \emph{Sobolev metrics} which arise by putting derivatives in the $L^2$-metric (cf.~\cite{BauerHarmsMichor,MichorMumford}) and the \emph{weighted Sobolev metrics} which arise by putting both weights and derivatives in the $L^2$-metric (cf.~\cite{BauerHarmsMichor_SobolevII}).
In \cite{Schulz}, the curvature weighted metric, which is an almost local metric, was considered in shape optimization to formulate approaches for unconstrained shape optimization problems. The first Sobolev metric was used in \cite{Schulz2015a} to formulate gradient-based methods to solve PDE-constrained shape optimization problems. In \cite{Welker_diffeological}, the gradient-based results from \cite{Schulz2015a} are extended by formulating the covariant derivative with respect to the first Sobolev metric. Thanks to that derivative, a Riemannian shape Hessian with respect to the first Sobolev metric could be specified, which opens the door to formulating higher-order methods in space of smooth shapes. 
If Sobolev or almost local metrics are considered, one has to deal with strong formulations of shape derivatives. 
An intermediate and equivalent result in the process of deriving these expressions is the weak expression as already mentioned above.
These weak expressions are preferable over strong forms.
Not only does one save analytical effort, but one needs lower regularity for the weak expressions. Moreover, the weak expressions are typically easier to implement numerically.
However, in the case of the more attractive weak formulation, the shape manifold $B_e$ and the corresponding Sobolev or almost local metrics are not appropriate. 
One possible approach to use weak forms is addressed in \cite{SchulzSiebenbornWelker2015:2}, which considers Steklov--Poincar\'{e} metrics. 
In the following,  some of the main results related to this metric from \cite{SchulzSiebenbornWelker2015:2} are summarized in view of obtaining efficient optimization methods, also for shape optimization problems under uncertainty.
For a comparison of the approach resulting from considering the first Sobolev and the approach based on the Steklov--Poincar\'{e} metric, the reader is referred to \cite{SchulzSiebenborn2016,Welker2016,Welker_diffeological}.

The \emph{Steklov--Poincar\'{e} metric} is given by
	\begin{equation}\label{definition_Steklovmetric}
	\begin{split}
	g^S\colon H^{1/2}(u)\times H^{1/2}(u) & \to \mathbb{R},\\
	(v,w) &\mapsto 
	\int_{u} v\cdot(S^{pr})^{-1}w\ \ds.
	\end{split}
	\end{equation}
	Here $S^{pr}$ denotes the projected Poincar\'e--Steklov operator, which is given by
	 $$S^{pr}\colon  H^{-1/2}(u) \to H^{1/2}(u),\
	v \mapsto \textup{tr}(V)\cdot \operatorname{n}$$ 
	 with $\textup{tr}\colon  H^1_0(D,\mathbb{R}^d) \to H^{1/2}(u,\mathbb{R}^d)$ denoting the trace operator on Sobolev spaces for vector-valued functions and $V\in H^1_0(D,\mathbb{R}^d)$ solving the Neumann problem 
	\begin{equation*}\label{weak-elasticity-N2}
	a(V,W)=\int_{u} v \, (\textup{tr}(W) \cdot \textup{n})\, \textup{d}s  \quad \forall\hspace{.3mm}  W\in H^1_0(D,\mathbb{R}^d),
	\end{equation*}
	where $a\colon H_0^1(D,\R^d) \times H_0^1(D,\R^d) \rightarrow \R$ is a symmetric and coercive bilinear form. 
	Note that a Steklov--Poincar\'{e} metric depends on the choice of the bilinear form. Thus, different bilinear forms lead to various Steklov--Poincar\'{e} metrics.
To define a metric on $B_e$, the Steklov--Poincar\'e metric is restricted to   the mapping $g^S\colon T_u B_e \times T_u B_e \rightarrow \R$. 
	
Next,  the connection between $B_e$ equipped with the Steklov--Poincar\'{e} metric $g^S$ and shape calculus is stated.
As already mentioned, the shape derivative can be expressed in a weak and strong form under the assumptions of the Hadamard Structure Theorem. 
The Hadamard Structure Theorem actually states the existence of a scalar distribution $r$ on the boundary  of a domain. However, in the following, it is always assume that $r$ is an integrable function. In general, if $r\in L^1(u)$, then $r$ is obtained in the form of the trace on $u$ of an element of $W^{1,1}(D)$. This means that it follows from Hadamard Structure Theorem that the shape derivative can be expressed more conveniently as
\begin{equation}
\label{bound_form}
d^{\text{surf}}j(u)[W]:=\int_u r(s)\left(W(s)\cdot \operatorname{n}(s)\right)  \ds .
\end{equation}
In view of the connection between the shape space $B_e$ with respect to the Steklov--Poincar\'{e} metric $g^S$ and shape calculus,  $r\in \mathcal{C}^\infty(u)$ is assumed. 
In contrast, if the shape functional is a pure volume integral, the weak form is given by
\begin{equation}
\label{vol_form}
d^{\text{vol}}j(u)[W]:=\int_D  RW(x)\, \dx,
\end{equation}
where $R$ is a differential operator acting linearly on the vector field $W$.

\begin{definition}[Shape gradient w.r.t. Steklov--Poincar\'{e} metric]
	\label{Def:Steklov}
	Let $r\in \mathcal{C}^\infty(u)$ denote the function in the shape derivative expression \eqref{bound_form}. Moreover, let $S^{pr}$ be the projected Poincar\'e--Steklov operator. A representation $v\in T_{u} B_e\cong\mathcal{C}^\infty(u)$ of the shape gradient in terms of $g^S$ is determined by 
	\begin{equation*}
	\label{Steklov_r}
	g^S(v,w)=\left(r,w\right)_{L^2(u)} \quad \forall w\in \mathcal{C}^\infty(u),
	\end{equation*}
	which is equivalent to
	\begin{equation}
	\label{equation_sgSteklov}
	\int_{u} w(s)\cdot [(S^{pr})^{-1}v](s) \d s=\int_{u} r(s)w(s) \d s \quad  \forall w\in \mathcal{C}^\infty(u).
	\end{equation}
\end{definition}

From \eqref{equation_sgSteklov}, one gets that a vector $V\in H_0^1(D,\R^d)\cap \mathcal{C}^\infty (D,\R^d)$ can be viewed as an extension of a Riemannian shape gradient to the hold-all domain $D$ because of the identities 
\begin{equation}
\label{Steklov_identity}
g^S(v,w)= d^{\text{surf}}j(u)[W]=a(V,W)\quad \forall W\in H_0^1(D,\R^d)\cap \mathcal{C}^\infty (D,\R^d),
\end{equation}
where $v=\text{tr}(V) \cdot \textup{n},w=\text{tr}(W) \cdot \textup{n}\in T_{u}B_e$. 
Since the strong formulation of the shape derivative arises from the weak formulation under the assumptions of the Hadamard Structure Theorem, one could also choose $d^\textup{vol}j(u)[W]$ in \eqref{Steklov_identity}.
This fact together with  identity \eqref{Steklov_identity} 
allows one to consider weak expressions of shape derivatives to compute the shape gradient with respect to $g^S$. Since both expressions of the shape derivative can be used, only $dj(u)[W]$ is written in the following.
	In order to compute the shape gradient, 
one has to solve the so-called \emph{deformation equation}
\begin{equation}
a(V, W) = dj(u)[W] \quad \forall W\in H_0^1(D,\R^d)\cap \mathcal{C}^\infty (D,\R^d).
\label{deformatio_equation}
\end{equation}
One option for $a(\cdot, \cdot)$ is the bilinear form associated with linear elasticity, i.e.,
\begin{equation*}
\label{eq:linear-elasticity-bilinear-form}
a^{\text{elas}}(V, W):=\int_D (\lambda \text{tr}  ( \epsilon(V)  )\text{id} + 2  \mu \epsilon(V)) : \epsilon(W)\, \dx,
\end{equation*}
where $\epsilon(W):= \frac{1}{2} \, (\nabla W + \nabla W^T)$, $A : B$ denotes the Frobenius inner product for two matrices $A, B$ and $\lambda,\mu \in \R$ denote the so-called Lam\'{e} parameters.

\begin{remark}
\label{remark:shape_gradient_not_longer_C_infty}
	Note that it is not ensured that $V\in H^1_0(D,\mathbb{R}^d)$ solving the PDE (in weak form)
	\begin{equation*}
	a(V, W) = dj(u)[W] \quad \forall W\in H_0^1(D,\R^d)
	\end{equation*}
	 is $\mathcal{C}^\infty(D,\R^d)$. Thus, $v=S^{pr}r=(\tr V)\cdot \operatorname{n}$ is not necessarily an element of $T_uB_e$.
	However, under special assumptions depending on the coefficients of a second-order partial differential operator and the right-hand side of the PDE, 	a weak solution $V$ that is at least $H^1_0$-regular is $\mathcal{C}^\infty$ (cf.~\cite[Section~6.3, Theorem~6]{Evans1998}).
\end{remark}

Thanks to the definition of the gradient with respect to $g^S$, algorithm~\ref{Algo:Gradient} can be applied on $(B_e,g^S)$.
In order to be in line with the above theory, it is assumed in algorithm~\ref{Algo:Gradient} that in each iteration $k$, the shape $u^k$ is a subset of the hold-all domain $D$. 
The Riemannian shape gradient is computed with respect to $g^S$ from \eqref{deformatio_equation}. The negative solution $-v=-\operatorname{tr}V\cdot \operatorname{n}$ is then used as descent direction for the objective functional $j$. 
The exponential map is used to update the shape iterates in algorithm~\ref{Algo:Gradient}.
Instead of the exponential map, it is also possible to use the concept
 of a retraction; this is a smooth mapping $\mathcal{R}\colon T\mathcal{U}\to \mathcal{U}$ satisfying $\mathcal{R}^{u^k}(0_{u_k})=u_k$ and the so-called local rigidity condition $\mathcal{R}^{u^k}_\ast(0_{u_k})=\text{id}_{T_{u^k}\mathcal{U}}$, where  $\mathcal{R}^{u^k}$ denotes the restriction of $\mathcal{R}$ to $T_{u^k}\mathcal{U}$, $0_{u_k}$ is the zero element of $T_{u^k}\mathcal{U}$ and $\mathcal{R}^{u^k}_\ast(0_k)$ denotes the pushforward of $0_{u_k}\in T_{u^k}\mathcal{U}$ by $\mathcal{R}$.
An example of a retraction is
\begin{equation}
\label{retraction}
\mathcal{R}^{u^k}\colon T_{u^k}\mathcal{U} \to \mathcal{U},\, 
v \mapsto \mathcal{R}^{u^k}(v)\colon = u^k+v
\end{equation}
(cf.  \cite{SchulzWelker}).
The retraction is only a local approximation; for large vector fields, the image of this function may no longer belong to $B_e$. This retraction is closely related to the perturbation of the identity, which is defined for vector fields on the domain $D$.  Given a starting shape $u^{k+1}$ in the $k$-{th} iteration of algorithm \ref{Algo:Gradient}, the perturbation of the identity acting on the domain $D$ in the direction $V^k$, where $V^k$ solves~\eqref{deformatio_equation} for $u=u^k$, gives
\begin{equation}
\label{deformation}
 D(u^{k+1}) = \{x \in D \, | \, x = x^k - t^k V^k \}.
\end{equation}
As vector fields induced from solving~\eqref{deformatio_equation} have less regularity than is required on the manifold, it is worth mentioning 
that the shape $u^{k+1}$ resulting from this update could leave the manifold $B_e$. To summarize, either large or less smooth vector fields can contribute to the iterate $u^{k+1}$ leaving the manifold. One indication that the iterate has left the manifold would be that the curve $u^{k+1}$ develops corners. Another possibility is that the curve $u^{k+1}$ self-intersects. One way to avoid this behavior is by preventing the underlying mesh to break (meaning elements from the finite element discretization overlap). One can avoid broken meshes as long as the step-size is not chosen to be too large. 

\begin{figure}
	\begin{center}
	\begin{overpic}[width=.63\textwidth]{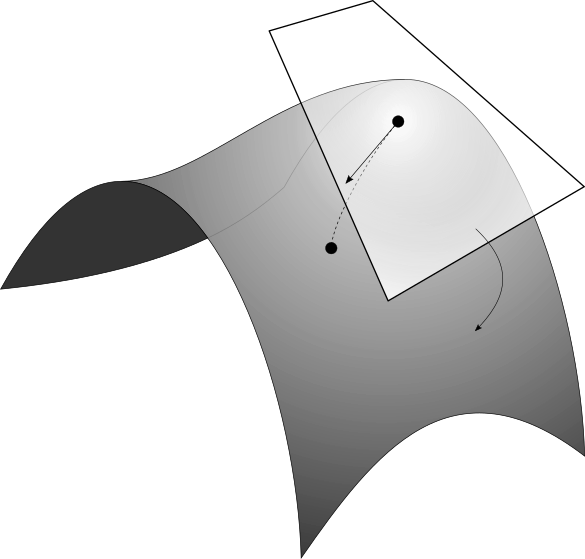}
	\put (55,88) {\large{$T_{u^k}\mathcal{U}$}}
	\put (70,77) {\large{$u^k$}}
	\put (65,66) {\large{$-t^kv^k$}}
	\put (54,46) {\large{$u^{k+1}$}}
	\put (74,44) {\large{$\exp_{u^k}$}}
	\put (52,15) {\large{$\mathcal{U}$}}
	\end{overpic}
	\end{center}
	\caption{Iterate $u^{k+1}= \exp_{u^k}(-t^kv^k)$, where $\exp_{u^k}\colon T_{u^k}\mathcal{U}\to \mathcal{U}$.}
	\label{fig:exponential}
\end{figure}
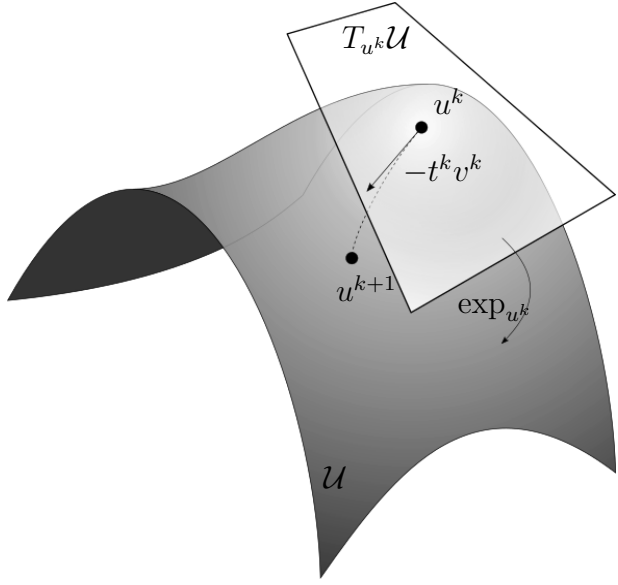

\begin{remark}
\label{remark:deformation_practice}
In practice, the hold-all domain is discretized by a mesh, for instance by finite elements (FE). Then in each iteration $k$, one computes the vector field $V^k$ defined on the hold-all domain by solving \eqref{deformatio_equation} for $u=u^k$. The vector field then informs how to move the computational mesh. For instance, with a FE discretization, $V^k$ acts on each node of the FE mesh, which moves not only the shape but also all other nodes of the mesh. An example of this is later shown in the application in figure~\ref{fig:deterministic-vector-fields}.
\end{remark}

\subsection{Optimization of multiple shapes}
\label{section:product-manifold}

This subsection extends algorithm \ref{Algo:Gradient} to multiple shapes $u=(u_1,\dots,u_N)\in\mathcal{U}^N$ with $N>1$ and $\mathcal{U}^N=\prod_{i=1}^{N}\mathcal{U}_i$ for Riemannian manifolds $(\cU_i,G^i)$. 
For this, the concepts of the pushforward, Riemannian shape gradient, and shape derivative needs to be generalized.
In view of applications in shape optimization, the metric $\mathcal{G}^N$ on the product manifold is related later to the  Steklov--Poincar\'e metric. As a main contribution,  the computation of vector fields extended to the hold-all domain is discussed. 

Analogously to \cite[3.3.12 Proposition]{abraham2012manifolds}, one can identify the tangent bundle $T\mathcal{U}^N$ with the product space $T\mathcal{U}_1\times\dots\times T\mathcal{U}_N$. 
In particular, there is an identification of the tangent space of the product manifold $\cU^N$ in the point $u$; more precisely, 
\begin{equation*}
\label{Multi-tangentspace}
T_u \cU^N \cong T_{u_1} \cU_1 \times \dots \times T_{u_N} \cU_N.
\end{equation*}
Let $\pi_i\colon \cU^N\to \cU_i$, $i=1, \dots, N$, be the $N$ canonical projections. With these identifications, one can then define the product metric $\mathcal{G}^N$ to the product shape space $\mathcal{U}^N$. For this, one needs the concept of the pushforward and the pullback by $\pi_i$.
	For each point $u\in \mathcal{U}^N$, the \emph{pushforward associated with canonical projections} $\pi_i$, $i=1, \dots, N$,  is given by the map
	$$(\pi_{i_\ast})_u\colon T_u\mathcal{U}^N\to T_{\pi_i(u)}\mathcal{U}_i,\, \mathfrak{c}\mapsto \frac{\d}{\d t} \pi_i(\mathfrak{c}(t)) \vert_{t=0}  =(\pi_i\circ \mathfrak{c})'(0).$$
	The \emph{pullback by the canonical projections} $\pi_i$, $i=1, \dots, N$, is the linear map from the space of 1-forms on $\mathcal{U}_i$ to the space of 1-forms on $\mathcal{U}^N$ and denoted by
	$$\pi_i^\ast\colon T^\ast_{\pi_i(u)}\mathcal{U}_i\to T^\ast_u\mathcal{U}^N,$$
	where  $T^\ast_{\pi_i(u)}\mathcal{U}_i$ and $T^\ast_u\mathcal{U}^N$ are the dual spaces of $T_{\pi_i(u)}\mathcal{U}_i$ and $T_u\mathcal{U}^N$, respectively.
Thanks to these definitions, the product metric $\mathcal{G}^N$ to the product shape space $\mathcal{U}^N$ can be defined:
\begin{equation*}
\mathcal{G}^N=\sum_{i=1}^N \pi_i^\ast G^i.
\end{equation*}
In particular, one has
\begin{equation}
\label{eq:product_metric}
\mathcal{G}^N_{u}(v,w) =\sum_{i=1}^{N}G_{\pi_i(u)}^{i}(\pi_{i_\ast}v,\pi_{i_\ast}w)\qquad\forall \,  v,w \in T_u \cU^N.
\end{equation}
Arguments identical to the ones in the proof of \cite[chapter 3, lemma 5]{o1983semi} make $(\mathcal{U}^N,\mathcal{G}^N)$ to a Riemannian product manifold.

In order to define a shape gradient of a functional $j\colon \mathcal{U}^N \rightarrow \R$ using the definition of the product metric in~\eqref{eq:product_metric},  definition \ref{Def:Pushforward} needs to be first generalized to the product shape space.

\begin{definition}[Multi-pushforward]
For each point $u\in \mathcal{U}^N$, the multi-pushforward associated with $J\colon\mathcal{U}^N \rightarrow \R$ is given by the map
\begin{equation*}
(j_\ast)_u\colon T_u\mathcal{U}^N\to \R,\, \mathfrak{c}\mapsto \frac{\d}{\d t} j(\mathfrak{c}(t)) \vert_{t=0}  =(j\circ \mathfrak{c})'(0).
\end{equation*}
\label{def:multi-shape}
\end{definition}

\begin{definition}[Riemannian multi-shape gradient]
The Riemannian multi-shape gradient for a shape functional $j\colon \mathcal{U}^N \rightarrow \R$ at the point $u=(u_1,\dots,u_N)\in\mathcal{U}^N$ is given by $v  \in T_u \cU^N$ satisfying
\begin{equation*}
\label{eq:shape_gradient_product_manifold}
\mathcal{G}^N_{u}\left(v,w \right) = (j_\ast)_uw \quad \forall \, w \in T_u \cU^N.
\end{equation*}
	\label{Def:MultipleShapeGradient}
\end{definition}

Notice that because of the identification of $T_u \cU^N$ with $T_{u_1} \cU_1 \times \dots \times T_{u_N} \cU_N$, the elements $\mathfrak{c}$ and $w$ from definitions~\ref{def:multi-shape} and \ref{Def:MultipleShapeGradient}, respectively, should be understood as vectors of the form $\mathfrak{c}(t) = (c_1(t),\ldots,c_N(t))$ and $w = (w_1,\ldots,w_N)$.

Thanks to the definition of the Riemannian multi-shape gradient, the steepest descent method on $(\mathcal{U}^N,\mathcal{G}^N)$ can be formulated (see algorithm \ref{Algo:multiples-shapes}). This method essentially follows the same steps as algorithm \ref{Algo:Gradient}. In algorithm \ref{Algo:multiples-shapes},  a \emph{multi-exponential map} 
\begin{equation}
\label{exp_multi}
\exp_{u^k}^N\colon T_{u^k}\mathcal{U}^N\to \mathcal{U}^N,\, z=(z_1,\dots,z_N)\mapsto (\exp_{u^k_1}z_1,\dots,\exp_{u^k_N}z_N)
\end{equation}
is needed
to update the shape vector $u^k=(u^k_1,\dots,u^k_N)$ in each iteration $k$, where
$\exp_{u^k_i}\colon T_{u^k_i}\mathcal{U}_i\to \mathcal{U}_i,\,z\mapsto \exp_{u^k_i}(z)$ for all $i=1,\dots,N$.
An Armijo backtracking line search strategy is used to calculate the step-size $t^k$ in each iteration. Here, the norm introduced on $\mathcal{G}^N$ is given by  $\|\cdot\|_{\mathcal{G}^N}:=\sqrt{\mathcal{G}^N(\cdot,\cdot)}$.

\begin{algorithm}
	\caption{Steepest descent method on $(\mathcal{U}^N,\mathcal{G}^N)$ with Armijo backtracking line search}
	\label{Algo:multiples-shapes}
	\begin{algorithmic}[0]
	\STATE \textbf{Require:} Objective function $j$ on $(\mathcal{U}^N,\mathcal{G}^N)$ 
\vspace{.1cm}
\STATE \textbf{Input:} Initial shape $u^0=(u^0_1,\dots,u^0_N)\in \mathcal{U}^N$\hspace*{15cm}\\
 
\hspace*{1cm} constants $\hat{\alpha}>0$ and $	{\displaystyle \sigma,\rho \in (0,1)}$ for Armijo backtracking strategy
\vspace{.3cm}

\STATE \textbf{for} $k=0,1,\dots$ \textbf{do}
\vspace{.1cm}
\STATE  [1] Compute the Riemannian multi-shape gradient $v^k$ with respect to $\mathcal{G}^N$ by solving
	\begin{equation}
\label{CompRiemShapeGradient_multi}
(j_\ast)_{u^k}w=G_{u^k}(v^k,w)\quad \forall\, w\in T_{u^k}\mathcal{U}^N.
\end{equation}
\STATE [2] Compute Armijo backtracking step-size: 
\vspace{.1cm}
\STATE \hspace*{1cm} Set $\alpha:=\hat{\alpha}$.
\STATE \hspace*{1cm} \textbf{while} $ j(\operatorname{exp}_{u^k}(-\alpha v^k)) > j(u^k)-\sigma\alpha \left\|v^k \right\|^2_{\mathcal{G}^N}$ 
\STATE \hspace*{1cm} Set $ \alpha :=\rho \alpha $.
\STATE \hspace*{1cm} \textbf{end while}
\STATE \hspace*{1cm} Set $t^k:=\alpha$.
\vspace{.1cm}
\STATE  [3] Set 
\begin{equation}
\label{Update_gS_multiple}
u^{k+1}:= \operatorname{exp}^N_{u^k}(-t^k v^k).
\end{equation}
\STATE \textbf{end for}
\vspace{.3cm}
	\end{algorithmic}
\end{algorithm}

So far in this subsection, each shape $u_i$ has been considered as an element of the Riemannian shape manifold $(\mathcal{U}_i,G^i)$, for all $i=1,\dots,N$, in order to define the multi-shape gradient with respect to the Riemannian metric $\mathcal{G}^N$.
In classical shape calculus, each shape $u_i$ is only a subset of $\R^d$. If one focuses on this perspective, then it is possible to generalize the classical shape derivative to a partial shape derivative and, thus, to a multi-shape derivative. With these generalized objects, a connection between shape calculus and the differential geometric structure of the product shape manifold $\mathcal{U}^N$ can be made.

Let $D$ be partitioned in $N$ non-overlapping Lipschitz domains $\Delta_1, \dots, \Delta_N$ such that $u_k \subset \Delta_k$. This construction will be referred as an \textit{admissible partition}. See figure \ref{fig:illustration_partition_D} for an example in $\R^2$. The indicator function $\mathbbm{1}_{\Delta_i}: D \rightarrow \{0,1\}$ is defined by $\mathbbm{1}_{\Delta_i}(x) = 1,$ if $x \in \Delta_i$, and $\mathbbm{1}_{\Delta_i}(x) = 0,$ otherwise.

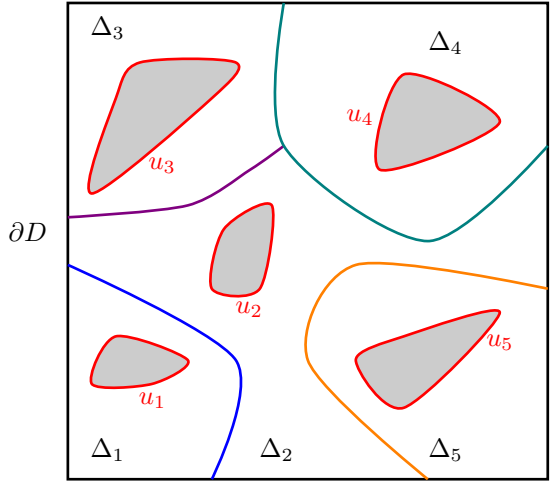
\begin{figure}
	\begin{center}
		\begin{overpic}[width=.5\textwidth]{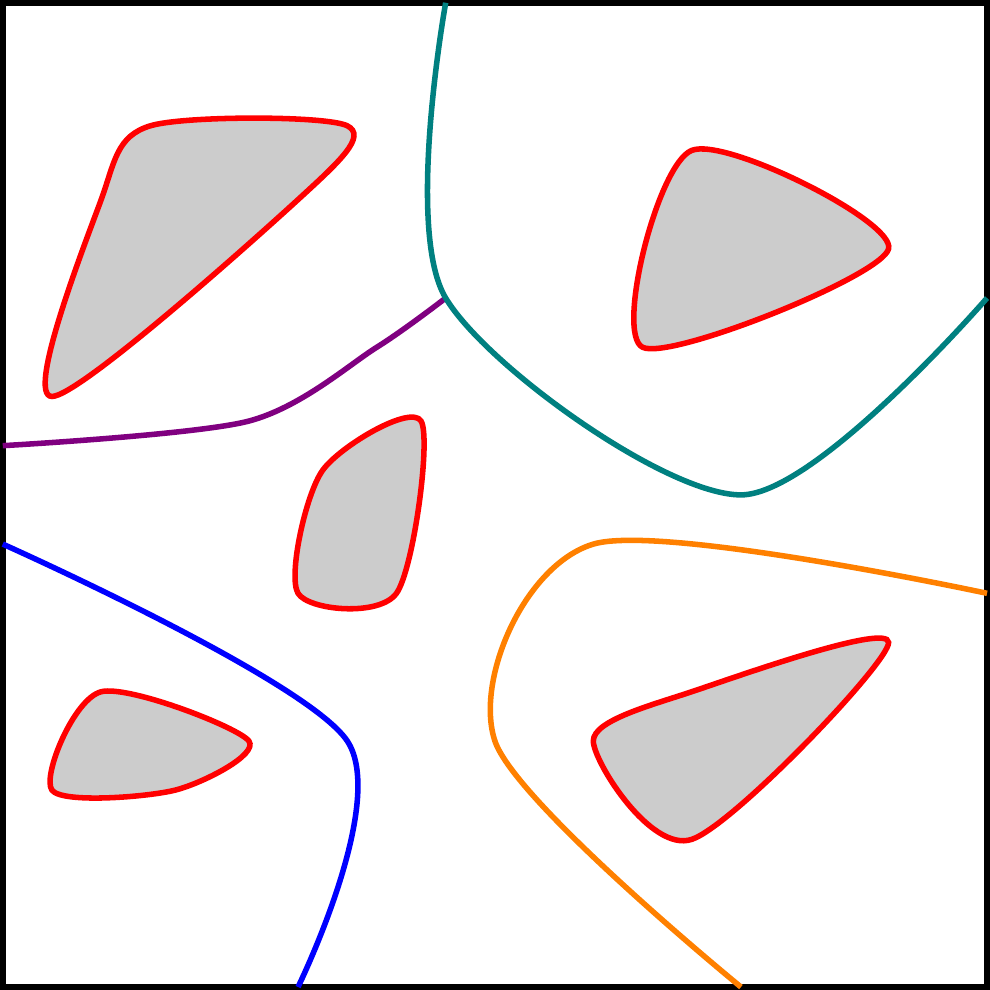}
			\put (-12,50){$\partial D$}
			\put (5,5){$\Delta_{1}$}
			\put (15,16){\textcolor{red}{$u_1$}}
			\put (35,35){\textcolor{red}{$u_2$}}
			\put (40,5){$\Delta_{2}$}
			\put (5,93){$\Delta_{3}$}
			\put (17,65){\textcolor{red}{$u_3$}}
			\put (75,90){$\Delta_{4}$}
			\put (58,75){\textcolor{red}{$u_4$}}
			\put (75,5){$\Delta_{5}$}
			\put (87,28){\textcolor{red}{$u_5$}}
		\end{overpic}
	\end{center}
	\caption{Illustration of a possible partition of $D\subset \R^2$.}
	\label{fig:illustration_partition_D}
\end{figure}

\begin{definition}[Multi-shape derivative]
Let $D\subset\R^d$ be open, $u=(u_1,\dots,u_N)$, and observe an arbitrary admissible partition with $u_i \subset \Delta_i$ for all $i=1,\dots,N$. Further, let $k\in\mathbb{N}\cup \{\infty\}$. For $i=1,\dots,N$, the $i$-th partial Eulerian derivative of a shape functional  $j$ at $u$ in direction $W\in\mathcal{C}^k_0(D,\mathbb{R}^d)$ is defined by
\begin{equation}
\label{eulerian_multi_partial}
d_{u_i}j(u)[\restr{W}{\Delta_i}]:=  \lim\limits_{t\to 0^+}\frac{j(u_1,\dots,u_{i-1},F_t^{\restr{W}{\Delta_i}}(u_i),u_{i+1},\dots,u_N)-j(u)}{t}.
\end{equation}
If for all directions $W\in\mathcal{C}^k_0(D,\mathbb{R}^d)$ the $i$-th partial Eulerian derivative  \eqref{eulerian_multi_partial} exists and the mapping 
\begin{equation*}
\mathcal{C}^k_0(D,\mathbb{R}^d)\to \mathbb{R}, \ W\mapsto d_{u_i}j(u)[\restr{W}{\Delta_i}]
\end{equation*}
is linear and continuous, the expression $d_{u_i}j(u)[\restr{W}{\Delta_i}]$ is called the $i$-th partial shape derivative of $j$ at $u$ in direction $W\in\mathcal{C}^k_0(D,\mathbb{R}^d)$.
If the $i$-th partial shape derivatives of $j$ at $u$ in the direction $W\in\mathcal{C}^k_0(D,\mathbb{R}^d)$ exist for all $i=1,...,N$, then
\begin{equation}
\label{eulerian_multi}
dj(u)[W]:= \sum_{i=1}^{N} d_{u_i}j(u)[\restr{W}{\Delta_i}]
\end{equation}
defines the multi-shape derivative of $j$ at $u$ in direction $W\in\mathcal{C}^k_0(D,\mathbb{R}^d)$.
\label{def_multi_shapeder}
\end{definition}

\begin{remark}
\label{rem:Hadamard-Thm-multiple}
For a single shape, by the Hadamard Structure Theorem, the shape derivative takes either the forms \eqref{bound_form} or \eqref{vol_form}. Using the definition above, the Hadamard Structure Theorem for multiple shapes can also be applied. The surface representation for $r_i \in L^1(u_i)$ is
\begin{equation}
\label{bound_form_multiple}
d^{\text{surf}}j(u)[W]:=\sum_{i=1}^N d^{\text{surf}}_{u_i} j(u)[\restr{W}{\Delta_i}]=\sum_{i=1}^N \int_{u_i} r_i(s)\left(\restr{W}{\Delta_i}(s)\cdot \operatorname{n}(s)\right)  \ds.
\end{equation}
The volume form is
\begin{equation}
\label{vol_form_multiple}
d^{\text{vol}}j(u)[W]: =  \sum_{i=1}^N  d^{\text{vol}}_{u_i}j(u)[\restr{W}{\Delta_i}]= \sum_{i=1}^N \int_{\Delta_i}  R_i \restr{W}{\Delta_i}(x)\, \dx,
\end{equation}
where $R_i$ is a differential operator acting linearly on the vector field $W$. In the volume form, it is clear that if $R_i = R$ for all $i$, the form \eqref{vol_form_multiple} reduces to
\begin{equation} 
\label{vol_form_multiple_reduced}
d^{\text{vol}}j(u)[W] = \int_D R W(x) \,\dx.
\end{equation}
\end{remark}

The expressions \eqref{bound_form_multiple} and \eqref{vol_form_multiple_reduced} suggest that the multi-shape derivative is in fact independent of the partition, provided it is an admissible one, i.e., with nonintersecting shapes and $u_i \subset \Delta_i$ for nonintersecting subdomains $\Delta_i$. This can be exploited computationally. It will be shown that to compute descent directions for the shape objective $j\colon \mathcal{U}^N \rightarrow \R$ according to \eqref{CompRiemShapeGradient_multi}, it is enough to solve the following variational problem: 
	\begin{equation}
	\label{eq:deformation_equation}
\text{find } V \in H_0^1(D,\R^d) \quad \text{such that}  \quad	a(V,W) =  dj(u)[W]\qquad\forall\, W\in H_0^1(D,\R^d). 
	\end{equation}
By virtue of remark~\ref{remark:shape_gradient_not_longer_C_infty}, the solution of~\eqref{eq:deformation_equation} is not necessarily $\mathcal{C}^\infty(D,\R^d)$, and these elements should be considered only formally.

In preparation for theorem~\ref{thm:equivalence_variational_formulations}, observe an admissible partition of $D$. The following Hilbert spaces are defined for all $ i=1\ldots, N$:
\begin{align*}
\bV_i &:=\{V \in H^1(\Delta_i,\R^d) \colon V = 0 \hbox{ on } \partial D \cap  \partial \Delta_i\}, \\
\bV_i^0 &= H_0^1(\Delta_i,\R^d).
\end{align*}
The following trace space for $\Gamma_i := \partial \Delta_i \backslash \partial D$ is defined:
\begin{equation*}
\Lambda_{i}:= \left\{ \eta \in H^{1/2}(\Gamma_i,\R^d) \colon  \eta = V\big|_{\Gamma_i}, \hbox{ for a suitable $V$ in }H_0^1(D,\R^d) \right\}.
\end{equation*}
One has (cf.~\cite[Subchapter 1.2]{quarteroni1999domain}) $\Lambda_i=H^{1/2}(\Gamma_i,\R^d)$ if $\Gamma_i \cap \partial D=\emptyset$. In  case $\Gamma_i \cap \partial D\not=\emptyset$, the space $\Lambda_{i}$ is strictly included in $H^{1/2}(\Gamma_i,\R^d)$, and is endowed with a norm which is larger than the norm of $H^{1/2}(\Gamma_i,\R^d)$. 
The trace space over $\Gamma:=\cup_{i=1}^N \Gamma_i$ is given by
\begin{equation*}
\Lambda:= \left\{ \eta \in H^{1/2}(\Gamma,\R^d) \colon  \eta = V\big|_{\Gamma}, \hbox{ for a suitable $V$ in }H_0^1(D,\R^d) \right\}.
\end{equation*}

The following main theorem justifies solving~\eqref{eq:deformation_equation} to obtain a vector field that gives descent directions with respect to each shape.
\begin{theorem}
\label{thm:equivalence_variational_formulations}
Observe an arbitrary admissible partition of $D$.
Suppose symmetric and coercive $a_i\colon \bV_i \times \bV_i \to \R $ are defined for all $i=1, \dots, N$ such that $a\colon H^1_0(D,\R^d)\times H^1_0(D,\R^d)\to \R$ satisfies $a(V,W)=\sum_{i=1}^{N} a_{i}(\restr{V}{\Delta_i},\restr{W}{\Delta_i})$ for all $V,W \in H^1_0(D,\R^d)$. 
Then the variational problem: find $V \in H_0^1(D,\R^d)$ such that
	\begin{equation}
	\label{eq:deformation_equation_new}
  \quad	a(V,W) =  dj(u)[W]\qquad\forall\, W\in H_0^1(D,\R^d)
	\end{equation}
 is equivalent to the system of variational problems: find $V_i \in \bV_i$, $i=1, \dots, N$ such that 
	\begin{subequations}
		\label{eq:multi_domain_weak_formulation}
	\begin{align}
	a_{i}(V_i,W_i) &= d_{u_i} j(u)[W_i] \qquad\forall\, W_i\in \bV_i^0, 	\label{eq:multi_domain_weak_formulation-a}\\ 
	V_i &= V_\ell                     \qquad\hbox{on all nonempty } \partial \Delta_i\cap \partial \Delta_\ell,	\label{eq:multi_domain_weak_formulation-b}\\
	\sum_{i=1}^N a_i(V_i,E_i \eta_i) &=  \sum_{i=1}^N d_{u_i} j(u)[E_i \eta_i]   \qquad \forall\,\eta \in \Lambda,	\label{eq:multi_domain_weak_formulation-c}
	\end{align}
		\end{subequations}
where $\eta_i=\eta|_{\Gamma_i}$ and $E_{i} \colon \Lambda_i \rightarrow \bV_i$ denotes an arbitrary extension operator, i.e., a continuous operator from $\Lambda_i$ to $\bV_i$ satisfying $(E_i \eta_i)|_{\Gamma_i} = \eta_i$.  
\end{theorem}

\begin{proof}
This proof follows the arguments from~\cite[Sec.~1.2]{quarteroni1999domain}, generalizing for the case $N>2$. First, it is shown that \eqref{eq:deformation_equation_new} yields the system \eqref{eq:multi_domain_weak_formulation}. Let $V$ be a solution to  \eqref{eq:deformation_equation}. Then setting $V_i = \restr{V}{\Delta_i}$ for $i=1, \dots, N$, one trivially obtains \eqref{eq:multi_domain_weak_formulation-b} in the sense of the corresponding traces. Moreover, using $W_i = \restr{W}{\Delta_i}$ for an arbitrary $W \in H_0^1(D,\R^d)$, one has $a_{i}(V_i,W_i) = d_{u_i} j(u)[W_i]$ 
for all $W_i\in \bV_i$, and in particular for all $W_i \in \bV_i^0$, showing \eqref{eq:multi_domain_weak_formulation-a}. Moreover, the function
\begin{equation}
\label{eq:extension-operator}
E \eta := \begin{cases}
 E_1 \eta_1  & \text{in } \Delta_1,\\
 \quad \vdots & \\
 E_N \eta_N & \text{in } \Delta_N
 \end{cases}
\end{equation}
belongs to $H_0^1(D,\R^d)$. In particular, one has
\begin{equation*}
a(V,E\eta) = dj(u)[E \eta],
\end{equation*}
which is equivalent to \eqref{eq:multi_domain_weak_formulation-c}.

Suppose now that $V_i$, $i=1, \dots, N$, are solutions to the system \eqref{eq:multi_domain_weak_formulation}. Let
\begin{equation*}
V: = \begin{cases}
V_1 & \text{in } \Delta_1,\\
\,\, \vdots & \\
V_N & \text{in } \Delta_N.
\end{cases}
\end{equation*}
From the condition $V_i = V_\ell$ on $\partial \Delta_i \cap \partial \Delta_\ell$, one obtains $V \in H_0^1(D,\R^d)$. Now, taking $W \in H_0^1(D,\R^d)$ gives $\eta:= \restr{W}{\Gamma} \in \Lambda$. Defining $E$ as in \eqref{eq:extension-operator} with $\eta_i=\restr{\eta}{\Gamma_i}$ yields $(\restr{W}{\Delta_i} - E_i \eta_i) \in \bV_i^0$ and hence \eqref{eq:multi_domain_weak_formulation-a} and \eqref{eq:multi_domain_weak_formulation-c} imply
\begin{align*}
a(V,W) &= \sum_{i=1}^N a_i(V_i, \restr{W}{\Delta_i} - E_i \eta_i) + a_i(V_i, E_i \eta_i)\\
&= \sum_{i=1}^N d_{u_i}j(u)[\restr{W}{\Delta_i} - E_i \eta_i] + d_{u_i}j(u)[E_i \eta_i]\\
&=dj(u)[W],
\end{align*}
meaning $V$ solves  \eqref{eq:deformation_equation}.
\end{proof}

\begin{remark}
There are several consequences of theorem~\ref{thm:equivalence_variational_formulations}. The first is computational: particularly for large-scale problems with many shapes, a decomposition approach can be used by solving \eqref{eq:multi_domain_weak_formulation} for an arbitrary admissible partition instead of the more expensive problem \eqref{eq:deformation_equation_new}. 
Second, for smaller-scaled problems, the theorem justifies the solving \eqref{eq:deformation_equation_new} ``all-at-once'' to obtain descent directions with respect to each shape. In particular, the solution $V_i$ to \eqref{eq:multi_domain_weak_formulation-a} gives a descent direction $-V_i$ for the shape $u_i$; due to the coercivity of $a_i$ one has 
\begin{equation*}
d_{u_i}j(u)[-V_i] = a_i(V_i, -V_i) < 0.
\end{equation*}
\end{remark}

\begin{remark}
The second and third conditions of~\eqref{eq:multi_domain_weak_formulation} are continuity conditions along $\Gamma$ for the solution $V$ and the normal flux (normal stress) relating $V_i$ for all $i=1,\ldots,N$. The extension operator $E_i$ can be chosen arbitrarily; one example is the extension-by-zero operator (cf.~\cite{hiptmair2015extension}).   
\end{remark}

Thanks to theorem \ref{thm:equivalence_variational_formulations}, the Riemannian multi-shape gradient with respect to $\mathfrak{g}^S:=\sum_{i=1}^{N}\pi_i^\ast g^S$ can be computed by solving \eqref{eq:deformation_equation} and, thus, algorithm \ref{Algo:multiples-shapes} can be applied on $(B_e^N,\mathfrak{g}^S)$.
In \eqref{Update_gS_multiple}, one can also consider a retraction mapping instead of the exponential map. If one chooses the retraction \eqref{retraction} instead of the exponential maps $\exp_{u^k_i}$ in \eqref{exp_multi}  for all $i=1,\dots,N$ in algorithm~\ref{Algo:multiples-shapes}, one gets again the relation to the perturbation of the identity. 
In this setting, theorem \ref{thm:equivalence_variational_formulations} justifies the update
\begin{equation}
\label{deformation_multiple}
D(u^{k+1}) = \{x \in D \, | \, x = x^k - t^k V^k \}
\end{equation}
with $u^{k+1}=(u^{k+1}_1,\dots, u^{k+1}_N)$ in the $k$-{th} iteration.

\begin{remark}
Notice that the variational problem given in~\eqref{eq:deformation_equation}, reflects exactly the approach presented, e.g., in~\cite{GeiersbachLoayzaWelker,SiebenbornNaegel,Siebenborn2017} to generate descent directions for problems containing multiple shapes. Hence the above theory supports the numerical approach already used in those papers.
\end{remark}

\section{Stochastic multi-shape optimization and the \\ stochastic gradient method}
\label{section:stochastic}

Given the framework for understanding shape optimization problems over product shape spaces, it is now possible to incorporate uncertainty.
In this section, the focus is on the case where the uncertainty can be characterized by a known probability space, for instance through prior sampling. The probability space is a triple $(\Omega, \mathcal{F}, \pP)$, where $\Omega$ is the sample space containing all possible ``realizations,'' $\mathcal{F} \subset 2^{\Omega}$ is the $\sigma$-algebra of events and $\pP\colon \Omega \rightarrow [0,1]$ is a probability measure. 

To account for uncertainty, it is natural to parameterize the corresponding objective, which now depends on the probability space. A parametrized shape functional is defined by a function
\begin{equation*}
J\colon \mathcal{U}^N \times \Omega \rightarrow \R, \,(u,\omega) \mapsto J(u,\omega).
\end{equation*}
Since $J$ depends on $\omega$, it is itself a random variable. To make the parameterized objective amenable to optimization, the following quantity
$$\EE[J(u,\cdot)]\colon = \int_{\Omega} J(u,\omega)\, \d \pP(\omega),$$
is used, i.e., the expectation or average. Other transformations of the parameterized objective are possible, for instance by use of disutility functions or risk functions; see \cite{Shapiro2009} for an introduction. A \textit{stochastic unconstrained shape optimization problem} is given by
\begin{equation}
\label{eq:SO-problem-abstract-2}
\min_{u \in \mathcal{U}^N} j(u):=\EE[J(u,\cdot)].
\end{equation}
Notice that the function $j$ representing the transformed function $J$ only depends on $u$, the vector of shapes. Therefore minimizers of \eqref{eq:SO-problem-abstract-2} do not depend on $\omega$, i.e., they are deterministic.

More interesting problems involve uncertainty in the equality constraint. The equality can be parametrized by the operator $e\colon \mathcal{U}^N \times \mathcal{Y} \times \Omega \rightarrow \mathcal{W}$, with Banach spaces $\mathcal{Y}$ and $\mathcal{W}$. A property is said to hold almost surely (a.s.) provided that the set in $\Omega$ where the property does not hold is a null set. Of interest are constraints of the form 
\begin{equation*}
e(u,y,\omega) = 0 \quad \text{a.s.}
\end{equation*}
In other words, $\pP(\{ \omega \in \Omega: e(u,y,\omega) \neq 0\}) = 0$. The solution 
$y=y(\omega)$ of this equation is a \textit{random state variable}. In applications, this belongs to the Bochner space $L^p(\Omega,\mathcal{Y})$, which given $p \in [1,\infty)$, is defined to be the set of all (equivalence classes of) strongly measurable functions $y\colon \Omega \rightarrow \mathcal{Y}$ having finite norm, where the norm is defined by $$\lVert y \rVert_{L^p(\Omega,\mathcal{Y})}:=(\EE[\lVert y \rVert_{\mathcal{Y}}^{p}])^{1/p} =\left(\int_\Omega \lVert y(\omega)\rVert_\mathcal{Y}^p \,\d \pP(\omega)\right)^{1/p}.$$ 
Letting the objective function depend on the state, a shape functional $\hat{J}\colon\mathcal{U}^N \times L^p(\Omega,\mathcal{Y}) \times \Omega \rightarrow \R$ is defined. With that, a \textit{constrained stochastic shape optimization problem} of the form
\begin{equation}
\label{eq:SO-problem-abstract}
\begin{aligned}
&\min_{u \in \mathcal{U}^N, y \in L^p(\Omega,\mathcal{Y})}\EE[\hat{J}(u, y(\cdot),\cdot)] \\
&\qquad\,\,\, \text{s.t.} \quad e(u,y,\omega) = 0 \quad \text{a.s.}
\end{aligned}
\end{equation}
is obtained.
If the equality constraint in \eqref{eq:SO-problem-abstract} is uniquely solvable for any choice of $u \in \mathcal{U}^N$ and almost every $\omega \in \Omega$, then the operator $S(\omega)\colon \mathcal{U}^N \rightarrow \mathcal{Y}, u \mapsto y(\omega)$ is well-defined for almost every $\omega$. As before, with $J(u,\omega):=\hat{J}(u,S(\omega)u,\omega)$, \eqref{eq:SO-problem-abstract} is formally equivalent to the problem \eqref{eq:SO-problem-abstract-2}. This unconstrained view will be helpful in formulating the stochastic gradient method. However, the reader is reminded that the stochastic gradient implicitly depends on the operator $S(\cdot)$.

If the stochastic dimension is relatively small, the expectation can be approximated using quadrature and algorithm \ref{Algo:multiples-shapes} can be applied. This type of \textit{sample average approximation} approach  is not an algorithm, and it becomes intractable as the stochastic dimension grows. For larger stochastic dimensions, the stochastic gradient method is widely used in stochastic optimization. It is a classical method developed by Robbins and Monro \cite{Robbins1951}.  As a sample-based approach, the stochastic gradient method does not suffer from the curse of dimensionality the way the discretizations mentioned in the introduction do. In~\cite{GeiersbachLoayzaWelker}, the stochastic gradient method was applied to the novel setting of shape spaces, where an example with multiple shapes was also presented. However, a theoretical background over product manifolds was not considered there.
To apply the method to the setting containing multiple shapes, several concepts developed in subsection~\ref{section:product-manifold} need to be generalized. To this end, it will sometimes be helpful to use the shorthand $J_\omega(\cdot):=J(\cdot,\omega)$.

\begin{definition}[Multi-pushforward for a fixed realization]
For each point $u\in \mathcal{U}^N$, the multi-pushforward associated with $J\colon\mathcal{U}^N \times \Omega \rightarrow \R$ for a fixed realization $\omega \in \Omega$ is given by the map
$$((J_\omega)_\ast)_u\colon T_u\mathcal{U}^N\to \R,\, \mathfrak{c}\mapsto \frac{\d}{\d t} J_\omega(\mathfrak{c}(t)) \vert_{t=0} =(J_\omega\circ \mathfrak{c})'(0).$$
\end{definition}

\begin{definition}[Stochastic Riemannian multi-shape gradient]
The Riemannian multi-shape gradient for a parametrized shape functional $J\colon \mathcal{U}^N \times \Omega \rightarrow \R$ at the point $u=(u_1,\dots,u_N)\in\mathcal{U}^N$ is given by $v=v(\omega)  \in T_u \cU^N$ satisfying
\begin{equation*}
\label{eq:shape_gradient_product_manifold_stoch}
\mathcal{G}^N_{u}\left(v,w \right) = ((J_\omega)_\ast)_uw \quad \forall \, w \in T_u \cU^N.
\end{equation*}
	\label{Def:StochMultipleShapeGradient}
\end{definition}

\noindent
Now, definition \ref{def_multi_shapeder} is generalized to incorporated uncertainties.

\begin{definition}[Multi-shape derivative for a fixed realization]
Let $D\subset\R^d$ be open, $u=(u_1,\dots,u_N)$, and observe an arbitrary admissible partition with $u_i \subset \Delta_i$ for all $i=1,\dots,N$. Further, let $k\in\mathbb{N}\cup \{\infty\}$. 
For $i=1,\dots,N$, the $i$-th partial Eulerian derivative of a shape functional  $J$ at $u$ for a fixed realization $\omega\in \Omega$ in direction $W\in\mathcal{C}^k_0(D,\mathbb{R}^d)$ is defined by
\begin{equation}
\label{eulerian_multi_partial_stoch}
d_{u_i}J(u,\omega)[\restr{W}{\Delta_i}]:=  \lim\limits_{t\to 0^+}\frac{J(u_1,\dots,u_{i-1},F_t^{\restr{W}{\Delta_i}}(u_i),u_{i+1},\dots,u_N,\omega)-J(u,\omega)}{t}
\end{equation}
If for all directions $W\in\mathcal{C}^k_0(D,\mathbb{R}^d)$ the $i$-th partial Eulerian derivative  \eqref{eulerian_multi_partial_stoch} exists and the mapping 
\begin{equation*}
\mathcal{C}^k_0(D,\mathbb{R}^d)\to \mathbb{R}, \ W\mapsto d_{u_i}J(u,\omega)[\restr{W}{\Delta_i}]
\end{equation*}
is linear and continuous, the expression $d_{u_i}J(u,\omega)[\restr{W}{\Delta_i}]$ is called the $i$-th partial shape derivative of $j$ at $u$ in direction $W\in\mathcal{C}^k_0(D,\mathbb{R}^d)$.
If the $i$-th partial shape derivatives of $J$ at $u$ for a fixed realization $\omega\in\Omega$ in the direction $W\in\mathcal{C}^k_0(D,\mathbb{R}^d)$ exist for all $i=1,...,N$, then

\begin{equation}
\label{eulerian_multi_stoch}
dJ(u,\omega)[W]:= \sum_{i=1}^{N} d_{u_i}J(u,\omega)[\restr{W}{\Delta_i}]
\end{equation}
defines the multi-shape derivative of $J$ at $u$ for a fixed realization $\O\omega\in \Omega$ in direction $W\in\mathcal{C}^k_0(D,\mathbb{R}^d)$.
\label{def_multi_shapeder_stoch}
\end{definition}

Using identical arguments to those in \cite[Lemma 2.14]{GeiersbachLoayzaWelker}, it is possible to show under what conditions  $j$ is shape differentiable in $u$. 
\begin{lemma}
\label{lemma:differentiability-expectation}
Suppose that $J(\cdot,\omega)$ is shape differentiable in $u$ for almost every $\omega \in \Omega$. Assume there exists a $\tau > 0$ and a $\PP$-integrable real function $C\colon \Omega \rightarrow \R$ such that for all $t \in [0,\tau]$, all $W \in C_0^\infty(D,\R^d)$, all $i = 1,\ldots, N$, and almost every $\omega$,
\begin{equation*}
\label{eq:difference-quotient-random-shape-derivative}
\frac{J(u_1,\dots,u_{i-1},F_t^{\restr{W}{\Delta_i}}(u_i),u_{i+1},\dots,u_N,\omega)-J(u,\omega)}{t} \leq C(\omega).
\end{equation*}
Then $j$ is shape differentiable in $u$ and 
\begin{equation*}
\label{eq:differentiability-expectation-exchange}
dj(u)[W] = \EE[d J(u,\cdot)[W]] \quad \forall W\in C_0^\infty(D,\R^d).
\end{equation*}
\end{lemma}

Equipped with these tools, it is now possible to formulate the stochastic gradient method for objectives formulated on a product shape space in algorithm \ref{Algo:stochastic_descent}. Instead of a backtracking procedure as in algorithm \ref{Algo:multiples-shapes} to determine the step-size, the algorithm uses the classical ``Robbins--Monro'' step-size from the original work \cite{Robbins1951}:
\begin{equation}
 \label{eq:Robbins-Monro-step-sizes}
 t^k \geq 0, \quad \sum_{k=0}^\infty t^k = \infty, \quad \sum_{k=0}^\infty (t^k)^2 < \infty.
\end{equation}
Under additional assumptions on the manifold and function $J$ (cf.~\cite{GeiersbachLoayzaWelker}), this rule guarantees step-sizes that are large enough to converge to stationary points while asymptotically dampening oscillations in the iterates. In contrast to the backtracking procedure, the step-size sequence is in practice chosen exogenously and its scaling is either informed by a priori estimates or tuned offline.
 
\begin{algorithm}
		\caption{Stochastic gradient method on $(\mathcal{U}^N,\mathcal{G}^N)$ with Robbins--Monro step-size}
	\label{Algo:stochastic_descent}
		\begin{algorithmic}[0]
			\STATE \textbf{Require:} Objective function $J$ on $(\mathcal{U}^N,\mathcal{G}^N)$
			\vspace{.1cm}
			\STATE \textbf{Input:} Initial shape $u^0=(u^0_1,\dots,u^0_N)\in \mathcal{U}^N$ 
			\vspace{.3cm}
			
			\STATE \textbf{for} $k=0,1,\dots$ \textbf{do}
			\vspace{.1cm}
			\STATE[1] Randomly sample $\omega^k$, independent of $\omega^1, \dots, \omega^{k-1}$
			\STATE [2] Compute the stochastic Riemannian multi-shape gradient $v^k = v^k(\omega^k)$ w.r.t. $\mathcal{G}^N$ by solving
	\begin{equation*}
\label{CompRiemShapeGradient_multi_stoch}
((J_{\omega^k})_\ast)_{u^k}w=G_{u^k}(v^k,w)\quad \forall\, w\in T_{u^k}\mathcal{U}^N.
\end{equation*}
			\vspace{.1cm}
			\STATE [3] Set 
			\begin{equation*}
			\label{Update_gS_stochastic}
			u^{k+1}:= \operatorname{exp}^N_{u^k}(-t^k v^k))
			\end{equation*}
			\\\phantom{[3]} for a steplength $t^k$ satisfying \eqref{eq:Robbins-Monro-step-sizes}.
			\vspace{.1cm}
			\STATE \textbf{end for}
			\vspace{.3cm}
		\end{algorithmic}
\end{algorithm}

In algorithm~\ref{Algo:stochastic_descent}, a new random realization $\omega^k$ is generated at each iteration $k$. This is used to compute a stochastic gradient $v^k = v^k(\omega^k)$, which is then used as a descent direction for the objective functional $J(\cdot,\omega^k)$.  If $\omega^k$ comprises a single sample from the probability space, the computation of the descent direction $v^k$ is as cheap as in the deterministic case. Note that this is \textit{not} necessarily a descent direction for the ``true'' objective $j$, which in combination with the exogeneous step-size rule $t^k$ does not guarantee descent at each iteration. The exponential map is used to map back to manifold; see figure~\ref{fig:exponential-stochastic}.

\begin{figure}
	\begin{center}
	\begin{overpic}[width=.63\textwidth]{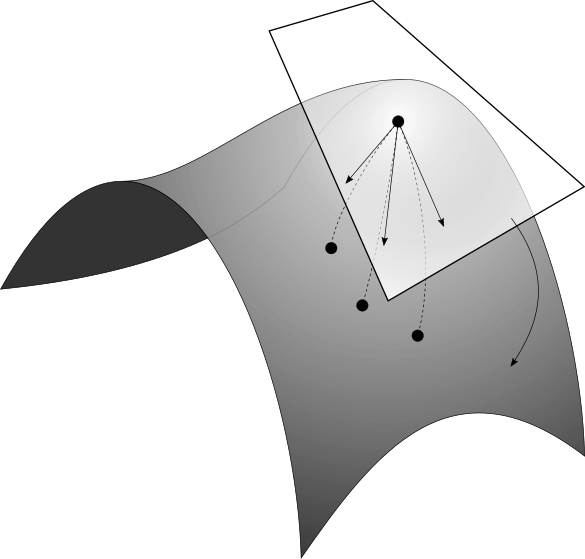}
	\put (53,88) {$T_{u^k}\mathcal{U}^N$}
	\put (70,77) {$u^k$}
	\put (73,66) {$-t^kv^k(\omega^{k,i})$}
	\put (60,35) {$u^{k+1}$}
	\put (82,44) {$\exp^N_{u^k}$}
	\put (52,15) {$\mathcal{U}^N$}
	\end{overpic}
	\end{center}
	\caption{Random iterates $u^{k+1}= \exp^N_{u^k}(-t^kv^k(\omega^{k,i}))$, where $\exp^N_{u^k}\colon T_{u^k}\mathcal{U}^N\to \mathcal{U}^N$.}
	\label{fig:exponential-stochastic}
\end{figure}

Some comments on possible improvements to the simple algorithm \ref{Algo:stochastic_descent} in the context of shape spaces are in order. One might ask whether a backtracking procedure could also be used for the stochastic setting; however, in \cite{Geiersbach2020}, it was demonstrated how the Armijo backtracking rule when combined with stochastic gradients fails in minimizing a function over the real line. Of course, there are modifications possible. In the most basic version of the method, $\omega^k$ comprises a single sample randomly drawn from the probability space. One might think that the problem could be remedied by simply taking multiple samples $\omega^k = (\omega^{k,1} , \dots, \omega^{k,m_k})$ at each iteration $k$ and computing the empirical average
\begin{equation}
\label{eq:empirical-stochastic-gradient}
\nabla J(u^k,\omega^k) = \frac{1}{m_k} \sum_{i=1}^{m_k} \nabla J(u^k, \omega^{k,i}).
\end{equation}
If $m_k$ is constant, then it is however easy to modify the example from \cite{Geiersbach2020} to show that simply taking more samples does not guarantee convergence of the method when paired with an Armijo backtracking procedure.  Asymptotic convergence results are known if one is ready to take $m_k \rightarrow \infty$, see \cite{Shapiro1996,Wardi1990}. 

Nevertheless, taking batches of samples like \eqref{eq:empirical-stochastic-gradient} is a simple way to reduce the variance of the gradient, and with that the iteration $u^{k+1}$. How the sampling sequence $\{m_k\}$ is to be chosen  strongly depends on the structure of the problem \eqref{eq:SO-problem-abstract} and the computational cost at each iteration $n$. In the context of optimal control problems with partial differential equations as constraints, one might additionally take into account that the computation is subject to numerical error as well. The authors in \cite{Martin2019} proposed a stochastic gradient step combined with a multilevel Monte Carlo scheme to reduce variance and numerical error. A method such as this one is sometimes referred to as a stochastic quasigradient method in the literature to emphasize the numerical bias induced by the iteration. The analysis in \cite{Martin2019}, which gives efficient choices for the sample size $m_k$, step-size $t_k$, and discretization error tolerance, works because the original problem is strongly convex, problem parameters are well-known, and the meshes involved are not deformed as part of the outer optimization loop. For more challenging problems, these choices no longer apply and future analysis would be needed.

Again for optimal control problems with PDEs, but for a larger class of problems, including nonsmooth and convex problems, the authors \cite{Geiersbach2020b} propose a different approximation scheme without needing to take additional samples (meaning $m_k \equiv 1$ is permissible). The proposed method uses averaging of the \textit{iterate} $u^k$ instead of the stochastic gradient. The descent is smoothed indirectly without having to take additional samples at each iteration. This was shown to work efficiently in combination with a mesh refinement rule, carefully coupled with the step-size rule $t^k$. Extending these results to the context of shape optimization would also be challenging as well, not only due to the analysis of numerical error and lack of convexity; here, $u^k$ represents a shape, not an element from a Banach space, and its ``average'' would need to be made precise. 

A final connection to the shape space $(B_e^N,\mathfrak{g}^S)$ is now desirable in view of the following numerical experiments. Using the theoretical justification from theorem~\ref{thm:equivalence_variational_formulations}, it is possible to compute a deformation vector $V = V(\omega) \in H_0^1(D,\R^d)$ in the point $u =(u_1, \dots, u_N) \in B_e^N$ by solving the variational problem
\begin{equation}
a(V,W) = dJ(u,\omega)[W] \quad \forall W \in H_0^1(D,\R^d).
\end{equation}
This deformation vector can be seen as an extension of the stochastic gradient $v=v(\omega)$ to the hold-all domain $D$. This stochastic deformation vector can then be used in the expression \eqref{deformation_multiple}.

\section{Numerical investigations}
\label{sec:numerical_experiments}

In this section, the shape optimization model is formulated in order to demonstrate the algorithms. The deterministic model is given in subsection~\ref{Subsec:ModelProblem}. Here, the focus is on a stationary version of the multi-shape model introduced in \cite{Siebenborn2017}. 
For the stochastic example in subsection~\ref{Subsec:ModelProblem-Stoch}, the model from \cite{GeiersbachLoayzaWelker} is used, with adjustments to include multiple shapes and random fields. 
There are several motivations for the models, for instance the identification of cellular structures in biology \cite{Siebenborn2017} or electrical impedance tomography 
\cite{dambrine2019incorporating}. In subsection~\ref{subsection:Numerics}, the results of the experiments are shown. In particular, the effectiveness and performance of algorithm~\ref{Algo:multiples-shapes} and algorithm~\ref{Algo:stochastic_descent} are demonstrated.
Moreover, an experiment on a single shape is done, which shows the robustness of a stochastic solution.

\subsection{Deterministic model problem}
\label{Subsec:ModelProblem}

Consider a partition of the domain $D$ into $N+1$ disjoint subdomains $D_i\subset D$ in such a way that  $(\sqcup_{i=0}^N D_i) \sqcup (\sqcup_{i=1}^Nu_i) = D$, where $u_i=\partial D_i$, $i=1, \dots, N$ and $\sqcup$ denotes the disjoint union. In particular, $D$ depends on $u$, i.e., $D=D(u)$. Note that this partition is a new construction that is related to the physical model, and is not to be confused with the arbitrary partition constructed in section~\ref{section:product-manifold}. For a given function $f\colon D \rightarrow\R$, $f_i$ denotes the restriction $\restr{f}{D_i}\colon D_i \rightarrow \R$. Additionally, $\mathbbm{1}_{D_i}$ denotes the indicator function of the set $D_i$, meaning $\mathbbm{1}_{D_i}(x) = 1$ if $x \in D_i$ and $\mathbbm{1}_{D_i}(x) = 0$ if $x \not\in D_i$.

Let $\bar{y} \in H^1(D)$ be the target distribution and $g \in L^2(\partial D)$ be a source term. The permeability coefficient is defined on each subdomain $D_i$ by $\kappa_i \in C^1(D_i).$ The shorthand $\kappa := \sum_{i=0}^N \kappa_i  \mathbbm{1}_{D_i}$ will be useful in representing this function in the weak form. 

In the following, the objective function
\begin{equation*}
j(u):= j^\text{obj}(u)+ j^\text{reg}(u)
\end{equation*}
with 
\begin{align}
\label{eq:tracking}
j^\text{obj}(u)& := \frac{1}{2} \int_{D}(y(x)-\bar{y}(x))^2  \dx = \frac{1}{2} \sum_{i=0}^N \int_{D_i}(y_i(x)-\bar{y}_i(x))^2  \dx,\\
\label{eq:regularization}
j^\text{reg}(u)& :=  \sum_{i=1}^N \nu_i \int_{u_i}  \d S
\end{align} 
is considered. The tracking-type functional \eqref{eq:tracking} gives the distance in $L^2(D)$ between the function $y$ and the target $\bar{y}$. 
In \eqref{eq:regularization}, $ \d S$ is used to characterize a surface integral. Note that the functional \eqref{eq:regularization} regularizes the perimeter with respect to each shape and different choices for  $\nu_i \geq 0$ can be made.

The following PDE-constrained problem in strong form is given:
\begin{align}
\min_{u \in B_e^N} \quad j(u) \hspace{2cm} & \label{Obj} \\
\text{s.t.} \quad - \nabla \cdot (\kappa_i(x) \nabla y_i(x)) &=  0\quad \text{in } D_i, \quad i=0,\dots,N,  \label{eq:PDE1d} \\
\kappa_0(x)\frac{\partial y_0}{\partial \text{n}_0}(x) &= g(x) \quad \text{on } \partial D, \, \label{eq:PDE2d}
\end{align}
where $\text{n}_0$ represents the outward normal vector on $D_0$. The equations \eqref{eq:PDE1d}--\eqref{eq:PDE2d} are complemented by the transmission conditions 
\begin{equation}\label{eq:jumpconditionsd}
\kappa_i(x) \frac{\partial y_i}{\partial \text{n}_i}(x) + \kappa_0(x) \frac{\partial y_0}{\partial \text{n}_0}(x)  = 0, \quad y_i(x) - y_0(x)  = 0 \quad  \text{on } u_i , \quad i=1, \dots, N.
\end{equation}
Note that the system \eqref{eq:PDE1d}--\eqref{eq:jumpconditionsd} can be compactly represented in the weak formulation: find $y \in H_{\text{av}}^1(D) := \{  v \in H^1(D) | \int_D v \,\dx = 0\}$ such that
\begin{equation*}
\int_D \kappa(x) \nabla y(x) \cdot \nabla v(x) \d x = \int_{\partial D} g(x) v(x) \d x \quad \forall v \in H_{\text{av}}^1(D).
\end{equation*}

\begin{remark}
Thanks to \cite[Proposition 3.1]{Ito-Kunisch-Peichl}, the regularity of $y_i$, $i=0,\dots,N$, is better than the one of $y$. More precisely, the solution $y\in H_{\text{av}}^1(D)$ of \eqref{eq:PDE1d}--\eqref{eq:jumpconditionsd} satisfies $y_i\in H^2(D_i)$, $i=0,\dots, N$. 
\end{remark}

\begin{remark}
In general, the distribution $\bar{y}$ and the diffusion coefficient $\kappa$  do not need to have as high a regularity as assumed above to formulate the PDE-constrained problem \eqref{Obj}--\eqref{eq:jumpconditionsd}. The regularity above is only needed for shape differentiability of the objective functional, see \cite[Section 3.2]{Ito-Kunisch-Peichl}. 
\end{remark}

The shape derivative to \eqref{Obj}--\eqref{eq:jumpconditionsd} can be achieved using standard calculation techniques like the one mentioned in subsection~\ref{section:DetOpt} combined with the help of the partial shape derivative definition and remark~\ref{rem:Hadamard-Thm-multiple}.
Its volume formulation is given by
\begin{equation}
\label{sd_j1d}
\begin{split}
d j(u)[W]=& 
\int_D 
-\kappa(x)\nabla y(x) \cdot (\nabla W(x)+\nabla W^\top(x))\nabla p (x)
\\
& \hspace{.4cm}- (y(x)-\bar{y}(x))\nabla\bar{y}(x) \cdot W(x) + (\nabla \kappa(x) \cdot W(x))\nabla y(x) \cdot \nabla p(x)
\\
&\hspace{.4cm} + \text{div}(W(x))\left( \frac{1}{2} (y(x)-\bar{y}(x))^2 + \kappa(x)\nabla y(x) \cdot \nabla p(x) \right)\,\dx\\
&+\sum_{i=1}^N \nu_i \int_{u_i} \mathfrak{v}_i(x) W(x) \cdot \text{n}_i(x)\,\, \d S,
\end{split}
\end{equation}
where $ \mathfrak{v}_i$ and denotes the curvature of the shape $u_i$,  $i=1,\dots,N$,  $ y(x)$ satisfies the \emph{state equation}~\eqref{eq:PDE1d}--\eqref{eq:jumpconditionsd} and $p(x)$ satisfies \emph{adjoint equation} given in strong form by
\begin{align}
- \nabla \cdot (\kappa_i(x) \nabla p_i(x)) &=  \bar{y}(x)-y_i(x)\quad \text{in } D_i, \quad i=0, \dots, N,  \label{eq:PDE1ad} \\
\kappa_0(x)\frac{\partial p_0}{\partial \text{n}_0}(x)& = 0 \quad \text{on } \partial D \, \label{eq:PDE2ad}
\end{align}
with the corresponding transmission conditions
\begin{equation}
\label{jumpdet}
\kappa_i(x) \frac{\partial p_i}{\partial \text{n}_i}(x) +\kappa_0(x) \frac{\partial p_0}{\partial \text{n}_0}(x)  = 0, \quad p_i(x)-p_0(x)  = 0\quad \text{on } u_i, \quad i=1, \dots, N.
\end{equation} 
The sum of integrals over $u_i$ in \eqref{sd_j1d}  is the shape derivative of the perimeter regularization, which is computed with the help of the partial shape derivative definition as follows:
$$
dj^\text{reg}(u)[W]=\frac{\d^+}{\d t}\,\rule[-2.5mm]{.1mm}{6mm}_{\hspace{.5mm}t=0}\, \sum_{i=1}^N v_i \int_{F_t^{\restr{W}{\Delta_i}}(u_i)}\d S ,
$$
where the $\ell$-th partial shape derivative of $j^\text{reg}$ at $u$ in direction $W$ is given by 
\begin{align*}
d_{u_\ell}j^\text{reg}(u)[\restr{W}{\Delta_\ell}]&=\frac{\d^+}{\d t}\,\rule[-2.5mm]{.1mm}{6mm}_{\hspace{.5mm}t=0}\, \left(\underset{i\not=\ell}{\sum_{i=1}^N} v_i\, \int_{u_i}\d S\right) +v_\ell\, \frac{\d^+}{\d t}\,\rule[-2.5mm]{.1mm}{6mm}_{\hspace{.5mm}t=0}\, \int_{F_t^{\restr{W}{\Delta_\ell}}(u_\ell)}\d S \\
& = v_j\, \frac{\d^+}{\d t}\,\rule[-2.5mm]{.1mm}{6mm}_{\hspace{.5mm}t=0}\, \int_{F_t^{\restr{W}{\Delta_\ell}}(u_\ell)}\d S =\int_{u_\ell} \mathfrak{v}_\ell(x) \restr{W}{\Delta_\ell}(x) \cdot \text{n}_\ell(x)\,\, \d S,
\end{align*}
where the last equality holds thanks to \cite[Proposition 5.1]{novruzi2002structure}. This gives the $\ell$-th partial shape derivative
$d_{u_\ell}j^\text{reg}(u)[\restr{W}{\Delta_\ell}]$ and thus the shape derivative of the regularization term in \eqref{sd_j1d}.

\smallskip
Now, every object needed for the application of algorithm~\ref{Algo:multiples-shapes} is given. In subsection~\ref{subsection:Numerics},  this algorithm is applied to solve the deterministic model problem.

\subsection{Stochastic model problem}
\label{Subsec:ModelProblem-Stoch}
For the stochastic model, the domain $D$ is partitioned as described for the deterministic model above. For a function $f\colon D \times \Omega \rightarrow \R$ the function $f_i$ denotes the restriction $f|_{D_i}\colon  D_i \times \Omega \rightarrow \R$. The slightly abusive notation $\nabla f_i(x,\omega) = \nabla_x f_i(x,\omega)$ means  $\omega$ is fixed and the gradient is to be understood with respect to the variable $x$ only. Additionally, the notation for the directional derivative means $\frac{\partial f_i}{\partial \text{n}_i}(x,\omega) = \lim_{t \rightarrow 0} \tfrac{1}{t}(f_i(x+t\operatorname{n}_i(x),\omega) - f_i(x,\omega)).$ 
A parametrized objective function is now given by
\begin{equation*}
\label{objective}
J(u,\omega):=J^\text{obj}(u, \omega)+ J^\text{reg}(u),
\end{equation*}
where 
\begin{align}
\label{Objective_TrackingType}
J^\text{obj}(u, \omega)  :=\frac{1}{2} \int_D (y(x,\omega)  - \bar{y}(x))^2 \, \dx =\frac{1}{2} \sum_{i=0}^N \int_{D_i}(y_i(x,\omega)-\bar{y}_i(x))^2  \dx
\end{align}
and $J^\text{reg}$ is defined as in \eqref{eq:regularization}. For simplicity, the source term $g$ and the target term $\bar{y}$ are deterministic with the same regularity as in the previous section. Suppose however that the source of uncertainty comes from the coefficients, i.e., $\kappa_i = \kappa_i(x,\omega)$ are random fields with regularity $\kappa_i \in L^2(\Omega, C^1(D_i))$. This leads to a modification of the deterministic problem:
\begin{align}
\min_{u \in B_e^N} \quad \left\lbrace j(u):=\EE \big[ J(u,\omega) \big] \right\rbrace \hspace{.3cm}&\label{eq:problem}\\
\text{s.t.} \quad 
- \nabla \cdot (\kappa_i(x,\omega) \nabla y_i(x,\omega)) &=  0\quad \text{in } D_i \times \Omega, \quad i=0, \dots, N, \label{eq:PDE1} \\
 \kappa_0(x,\omega) \frac{\partial y_0}{\partial \text{n}_0}(x,\omega) &= g(x) \quad \text{on } \partial D \times \Omega \label{eq:PDE2}
\end{align}
The following transmission conditions are also imposed: 
\begin{equation}\label{eq:jumpconditions}
\begin{aligned}
 \kappa_i(x,\omega) \frac{\partial y_i}{\partial \text{n}_i}(x,\omega)+ \kappa_0(x,\omega) \frac{\partial y_0}{\partial \text{n}_0}(x,\omega)&= 0\quad &\text{ on } u_i \times \Omega,\, i=1, \dots, N, \\
  y_i(x,\omega) - y_0(x,\omega) &= 0   \quad &\text{ on } u_i \times \Omega,\, i=1, \dots, N.
 \end{aligned}
\end{equation}

Using standard techniques for calculating the shape derivative (see~\cite[Appendix B]{GeiersbachLoayzaWelker}), the shape derivative in volume formulation for a fixed $\omega$ is given by 
\begin{equation*}
\label{sd_j1}
\begin{aligned}
d J&(u,\omega)[W]\\=& \int_D -\kappa(x,\omega)\nabla y(x,\omega) \cdot (\nabla W(x)+\nabla W^\top(x))\nabla p(x,\omega)  \\
& \hspace{.4cm} - (y(x,\omega)-\bar{y}(x))\nabla\bar{y}(x) \cdot  W(x)+ (\nabla \kappa(x,\omega) \cdot W(x))\nabla y(x,\omega) \cdot \nabla p(x,\omega)
\\
&\hspace{.4cm} + \text{div}(W(x))\left( \frac{1}{2} (y(x,\omega)-\bar{y}(x))^2 + \kappa(x,\omega)\nabla y(x,\omega)\cdot \nabla p(x,\omega) \right)\,\dx\\
&+ \sum_{i=1}^N \nu_i \int_{u_i} \mathfrak{v}_i(x) W(x) \cdot \text{n}_i(x)\, \d S,
\end{aligned}
\end{equation*}
where $y = y(x,\omega)$ satisfies the \emph{state equation}~\eqref{eq:PDE1}--\eqref{eq:jumpconditions} and $p = p(x,\omega)$ satisfies \emph{adjoint equation} 
\begin{align}
-   \nabla \cdot (\kappa_i(x,\omega) \nabla p_i(x,\omega)) &=  \bar{y}(x) -y_i(x,\omega),\quad \text{in }D_i \times \Omega, \quad i=0, \dots, N, \label{eq:adjointPDE1} \\
\kappa_0(x,\omega) \frac{\partial p_0}{\partial \text{n}_0}(x,\omega)  &= 0, \quad \quad\text{on } \partial D \times \Omega, \label{eq:adjointPDE2}
\end{align}
with corresponding interface conditions
\begin{equation}\label{eq:jumpconditions-adjoint}
\begin{aligned}
\kappa_i(x,\omega) \frac{\partial p_i}{\partial \text{n}_i}(x,\omega) + \kappa_0(x,\omega) \frac{\partial p_0}{\partial \text{n}_0}(x,\omega)&=0\quad & \text{ on } u_i \times \Omega, i=1, \dots, N, \\
 p_i(x,\omega) - p_0(x,\omega) &= 0 \quad & \text{ on } u_i \times \Omega, i=1, \dots, N.
\end{aligned}
\end{equation}

The construction of the coefficients $\kappa$ for the purpose of simulations requires some discussion. Karhunen--Lo\`eve expansions are frequently used to simulation random perturbations of a coefficient within a material and are also used in the experiments in subsection~\ref{subsection:Numerics}. Given a domain $\tilde{D}$, a (truncated) Karhunen--Lo\`eve expansion of a random field $a\colon \tilde{D} \times \Omega \rightarrow\R$ takes the form
\begin{equation*}
a(x,\omega) = \bar{a}(x) +  \sum_{k=1}^{m} \sqrt{\gamma_{k}} \phi_{k}(x) \xi_{k}(\omega),
\end{equation*}
where $\bar{a}\colon \tilde{D} \rightarrow \R$ and $\xi(\omega)=(\xi_{1}(\omega), \dots, \xi_{m}(\omega)) \in \R^{m}$ is a random vector. 
The truncation is done for the purposes of numerical simulation and the choice of $m$ should be informed by error analysis. The terms $\gamma_k$ and $\phi_k$ are eigenvalues and eigenfunctions that depend on the domain $\tilde{D}$. In particular, they are associated with the compact self-adjoint operator defined via the covariance function $C \in L^2(\tilde{D}\times \tilde{D})$ by $\mathcal{C}(\phi)(x) = \int_{\tilde{D}} C(x,y) \phi(y) \textup{d} y$ for all $x\in \tilde{D}.$ For general domains, formulas giving explicit representations of $\gamma_k$ and $\phi_k$ do not exist and need to be numerically computed. However, since the subdomains vary as part of the optimization procedure, their computation here would be extremely expensive. Moreover, from a modeling perspective, it seems more realistic that the model for uncertainty in a specific material is constructed beforehand using samples on a fixed domain $\tilde{D} \supset D_i$.
 Ideally $\tilde{D}$ should be much larger than $D_i$ to limit the effects of the boundary of the larger domain on the sample. Then, to approximate $\kappa_i$ on $D_i$, one can first produce a sample on the larger domain $\tilde{D}$ and then use its restriction on the domain $D_i$ for computations. To be more precise, one would first define over $\tilde{D}$
\begin{equation}
\label{eq:KL-expansion}
 \tilde{\kappa}_{i}(x,\omega) =\bar{\kappa}_{i}(x) + \sum_{k=1}^{m_i} \sqrt{\gamma_{i,k}} \phi_{i,k}(x) \xi_{i,k}(\omega),
\end{equation}
where $\bar{\kappa}: \tilde{D} \rightarrow \R$, $\xi_{i,k}(\omega)=(\xi_{i,1}(\omega), \dots, \xi_{i,m_{i}}(\omega)) \in \R^{m_i}$ is a random vector, and $\gamma_{i,k}$ and $\phi_{i,k}$ denote the eigenvalues and eigenfunctions that depend on the domain $\tilde{D}$. 
Finally, $\kappa_i = \tilde{\kappa}_i |_{D_i}$. The coefficient $\kappa$ over the domain $D$ is then stitched together by definition of
\begin{equation*}\label{eq:randomfield}
\kappa(x,\omega)= \kappa_0(x,\omega)+\sum_{i=1}^N \kappa_{i}(x,\omega)\mathbbm{1}_{D_i}(x).
\end{equation*}
An example of this construction is shown in the next subsection in figure~\ref{fig:random-field-realizations}.

\subsection{Numerical experiments}
\label{subsection:Numerics}
The purpose of this section is to demonstrate the behavior and performance of 
algorithm~\ref{Algo:multiples-shapes} and algorithm~\ref{Algo:stochastic_descent}. 
Simulations were run on FEniCS \cite{Alnes2015}.
For all experiments, the hold-all domain is set to $D=[0,1]^2$ and a mesh with 2183 nodes and 4508 elements is used.

For methods relying on mesh deformation, one challenge is to ensure that meshes maintain good quality and do not become destroyed over the course of optimization. 
Many techniques have been developed along the years to overcome this challenge. 
There is the option of remeshing, see for instance~\cite{MorinNochettoPaulettiVerani:2012:1,Sturm:2016:1,FepponAllaireBordeuCortialDapogny:2019:1}. 
Of course, one could also use mesh regularization techniques, space adaptivity, among others as described for example in~\cite{BaenschMorinNochetto:2005:1,DoganMorinNochettoVerani:2007:1}. 
There is also the possibility of projecting the descent directions onto the subspace of perturbation fields generated only by normal forces, inspired by the Hardamard structure theorem~\cite{EtlingHerzogLoayzaWachsmuth:2020:1}.
Recently, a simultaneous shape and mesh quality optimization approach based on pre-shape calculus has also been proposed~\cite{LuftSchulz:2020:1, LuftSchulz:2021:1}.
Another option is to consider the method of mappings and impose certain restrictions on the maps that preserve mesh quality, see~\cite{HaubnerSiebenbornUlbrich:2021:1, OnyshkevychSiebenborn:2021:1}.

In this chapter, the techniques developed in~\cite{SchulzSiebenbornWelker2015:2,SchulzSiebenborn2016}, are considered. 
As discussed in ~\cite{SchulzSiebenbornWelker2015:2}, an unmodified right-hand side of the discretized deformation equation leads to deformation fields causing meshes with bad aspect ratios. One possibility is to set the values of the shape derivative to zero if the corresponding element does not intersect with the shapes, i.e.,
$$dj(u)[W]=0\quad \forall W\text{ with }\text{supp}(W)\cap u_i=\emptyset, \quad i=1, \dots, N.$$
Additionally, following the ideas from \cite{SchulzSiebenborn2016}, at each iteration $k$, an additional PDE is solved to choose values for the Lam\'{e} parameters in the deformation equation. The parameter $\lambda$ is set to zero, and $\mu$ is chosen from the interval $[\mu_{\min},\mu_{\max}]$ such that it is decreasing smoothly from $u_i$, $i=1, \dots, N$, to the outer boundary $\partial D$. 
One possible way to model this behavior is to solve the Poisson equation
  \begin{align*}
  \nonumber\Delta \mu  &= 0\hspace{.9cm}   \text{ in } D_{i}, \quad i=0, \dots, N \\
  \label{eq:system_mu_elas}\mu &= \mu_{\max}  \hspace{0.4cm}\text{ on } u_i, \quad i=1, \dots, N,\\
  \nonumber\mu &= \mu_{\min} \hspace{0.5cm} \text{ on } \partial D.
  \end{align*} 
In all experiments, $\mu_{\min} = 10$ and $\mu_{\max} = 25$ is chosen.

\subsubsection{Deterministic case: behavior of algorithm \ref{Algo:multiples-shapes}}
\label{subsection:Num_det}


The deterministic shape optimization problem formulated in subsection~\ref{Subsec:ModelProblem} is considered to demonstrate the behavior of algorithm \ref{Algo:multiples-shapes}. 
For the numerical experiments, an example with two shapes is used, i.e., $N=2$, and the algorithm runs for 400 iterations.
The Neumann boundary condition in \eqref{Obj}--\eqref{eq:jumpconditionsd}  is set to $g=1000$ and the perimeter regularization is set to $\nu_1 = \nu_2 = 2\cdot 10^{-5}$.

In order to generate the target data $\bar{y}$ in the tracking-type objective functional \eqref{Obj}, a target shape vector $u^*=(u^*_1, u^*_2)$ is chosen, which is displayed in dotted lines in figure~\ref{fig:deterministic-shapes}.
The target shapes, i.e., an ellipse and a (non-convex) curved tube, are chosen so the configuration is non-symmetric, making their identification more difficult.
The permeability coefficients are assumed to be piecewise constant on each subdomain with the choices $\kappa_{0} = 1000$ for the outer domain $D_0^*$, 
$\kappa_1 = 7.5$ corresponding to the ellipse $D_1^*$, 
and $\kappa_2 = 5$ corresponding to the curved tube $D_2^*$. 
The data $\bar{y}$ is computed by solving the state equation \eqref{eq:PDE1d}--\eqref{eq:jumpconditionsd} on the target configuration $D^* = (\sqcup_{i=0}^2 D_i^*) \sqcup (\sqcup_{i=0}^2 u_i^*)$; see figure~\ref{fig:deterministic-y-bar}. 

Let $D^k = (\sqcup_{i=0}^2 D_i^k) \sqcup (\sqcup_{i=0}^2 u_i^k)$ be the configuration of the subdomains at iteration $k$.
The subdomains $D_i^k$ correspond to the different colors in figure~\ref{fig:deterministic-shapes}.
As for the computation for the target distribution, the coefficients are assumed to be piecewise constant on each subdomain with the choices $\kappa_{0} = 1000$ for the outer domain $D_0^k$, 
$\kappa_1 = 7.5$ corresponding to $D_1^k$, 
and $\kappa_2 = 5$ corresponding to  $D_2^k$. 
For the Armijo rule, the values $\hat{\alpha}=0.0175$, 
$\rho=0.9$, and $\sigma=10^{-4}$ are used. Since the algorithm is designed to deform the mesh, the initial step-size $\hat{\alpha}$ is scaled to be proportional to the maximal diameter of the elements, which is used as a heuristic solution to avoid mesh destruction.  
 Figure~\ref{fig:deterministic-shapes} shows the progression of the subdomains.  Within 400 iterations, one sees that the configuration $D^k$ obtained by the method comes quite close to the target. 
Figure~\ref{fig:deterministic-vector-fields} gives a visualization of the vector fields $V^k$ induced by solving the deformation equation \eqref{eq:deformation_equation}.  In figure~\ref{fig:deterministic-plots} one sees the decay of the objective function values and the $H^1$--norm of the deformation vector as a function of iteration number. The Armijo line search procedure ensures that $j(u^{k+1}) \leq j(u^k)$ for all $k$. The $H^1$--norm of the descent directions serves as a stationary measure, and the plots show decreasing as a function of the iterations.


\begin{figure}
  \begin{subfigure}[b]{0.5\linewidth}
  \centering
    \includegraphics[width=\textwidth]{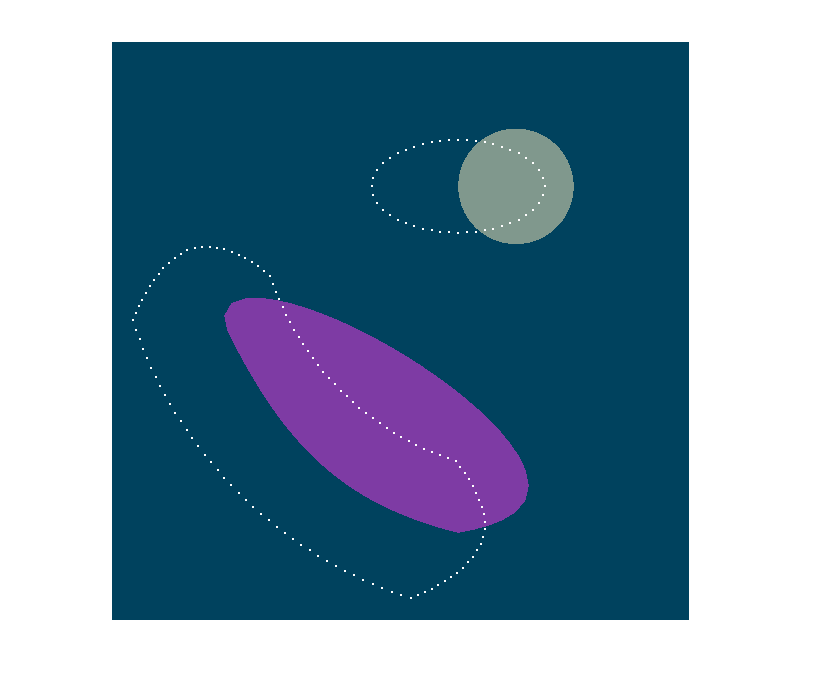}
    \caption{Initial configuration $D^0$}  
  \end{subfigure}%
  \begin{subfigure}[b]{0.5\linewidth}
    \centering
    \includegraphics[width=\textwidth]{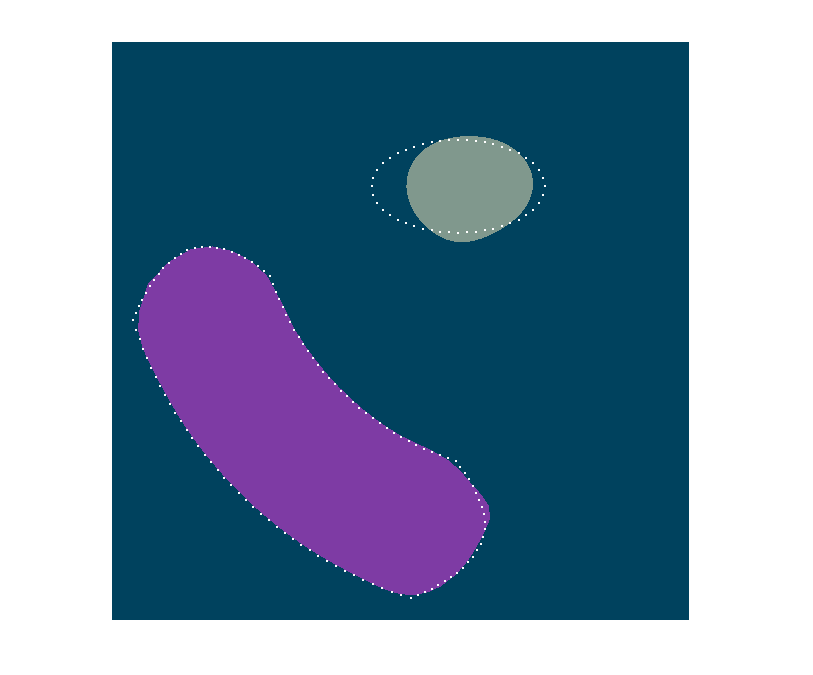}   
    \caption{$D^{50}$}
  \end{subfigure}\\%
  \begin{subfigure}[b]{0.5\linewidth}
    \centering
    \includegraphics[width=\textwidth]{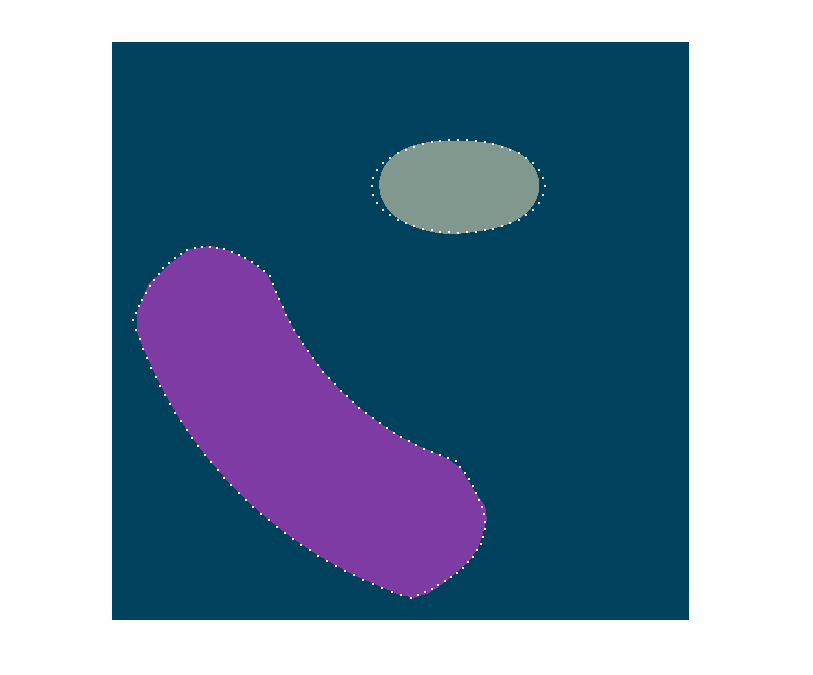}   		
   \caption{$D^{200}$}
  \end{subfigure}%
    \begin{subfigure}[b]{0.5\linewidth}
    \centering
    \includegraphics[width=\textwidth]{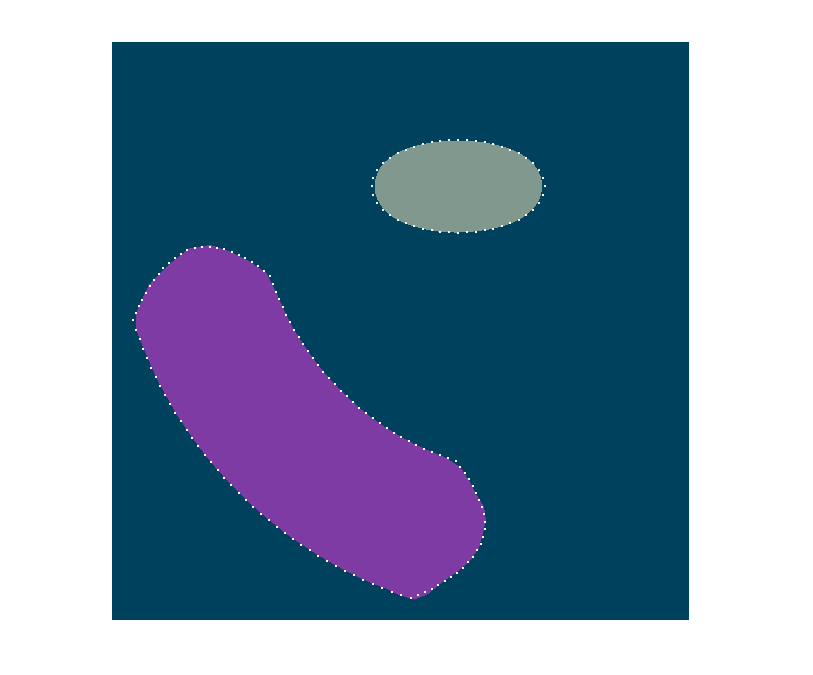}   		
   \caption{$D^{400}$}
  \end{subfigure}%
  \caption{The target shapes are displayed by the dotted lines. The outer domain $D_0^k$ is displayed in teal, the domain $D_1^k$ is displayed in light green, and the subdomain $D_2^k$ is shown in purple. The figures show the progression of the initial configuration $D^0$ to the final subdomain configuration $D^{400}$.}
  \label{fig:deterministic-shapes}
\end{figure}

  \begin{figure}
  \centering
    \includegraphics[width=0.5\textwidth]{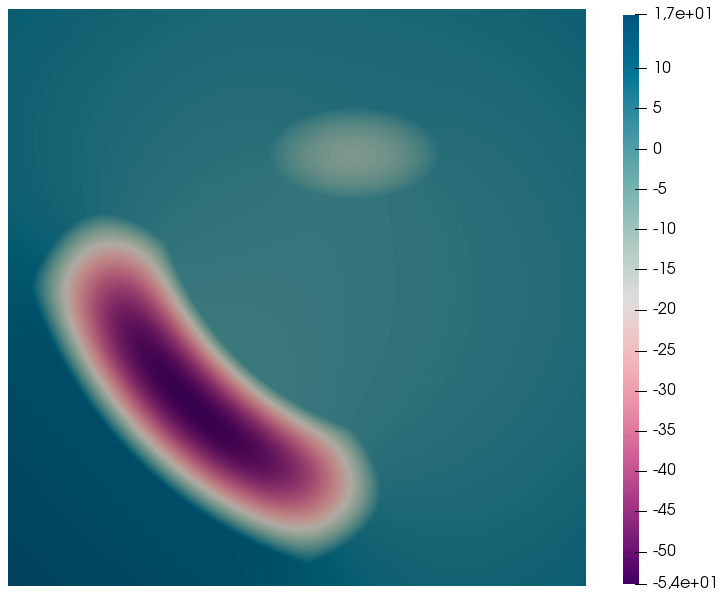}
\caption{Values of the target data $\bar{y}$.}
  \label{fig:deterministic-y-bar}
\end{figure}

 \begin{figure}
  \begin{subfigure}[b]{0.5\linewidth}
  \centering
    \includegraphics[width=\textwidth]{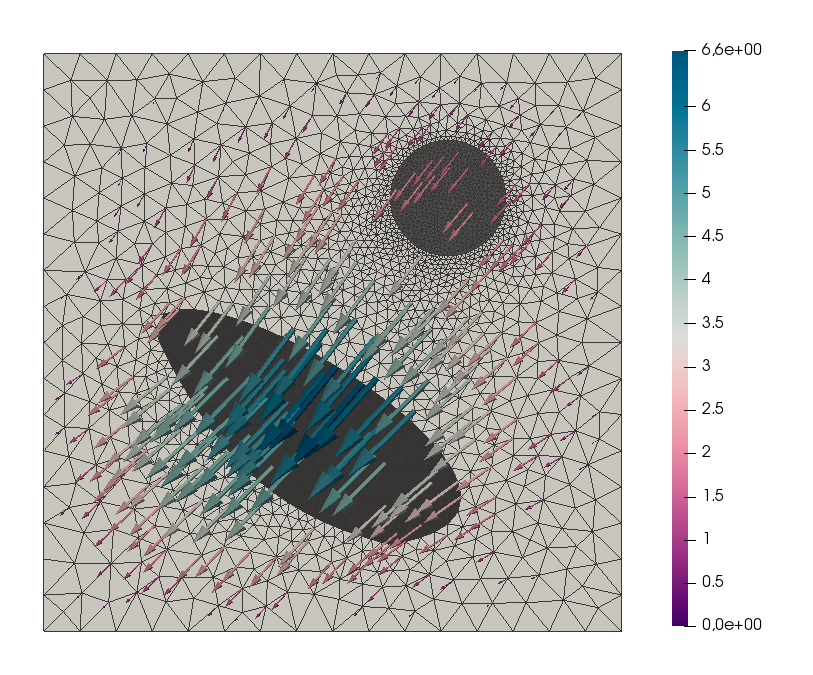}
    \caption{Vector field $V^0$}  
  \end{subfigure}%
  \begin{subfigure}[b]{0.5\linewidth}
    \centering
    \includegraphics[width=\textwidth]{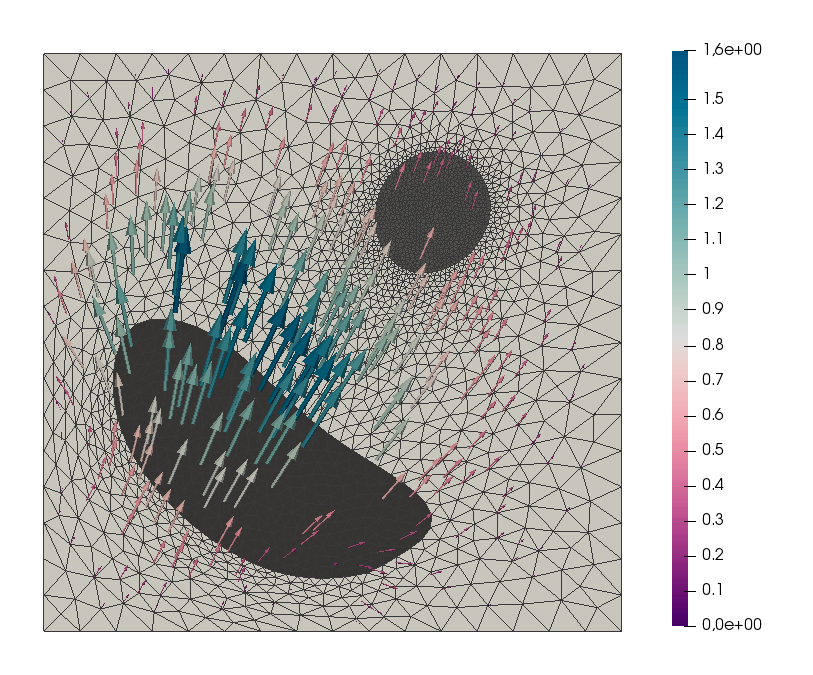}   
    \caption{Vector field $V^{3}$}
  \end{subfigure}
  \caption{Vector fields $V^k$ are displayed that result from solving the deformation equation	\eqref{eq:deformation_equation} at iteration $k$. }
  \label{fig:deterministic-vector-fields}
\end{figure}

 \begin{figure}
  \begin{subfigure}[b]{0.45\linewidth}
  \centering
    \includegraphics[width=\textwidth]{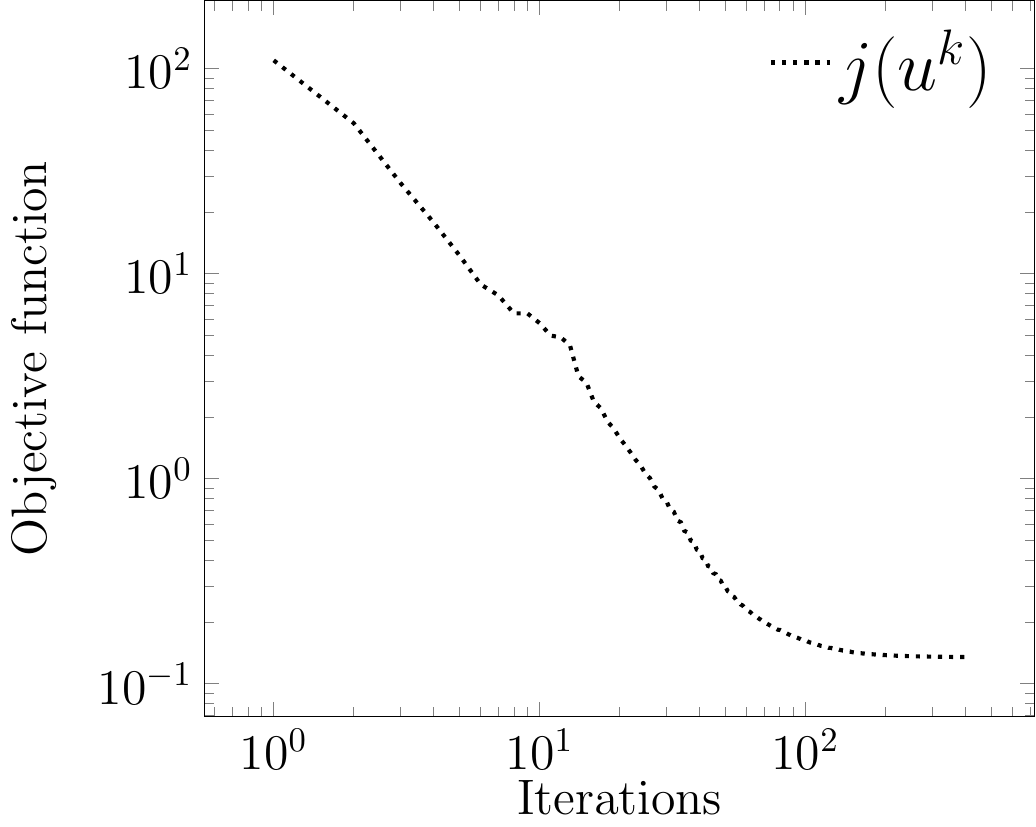} 
    \caption{Objective function decay}  
  \end{subfigure}%
  \begin{subfigure}[b]{0.45\linewidth}
    \centering
    \includegraphics[width=\textwidth]{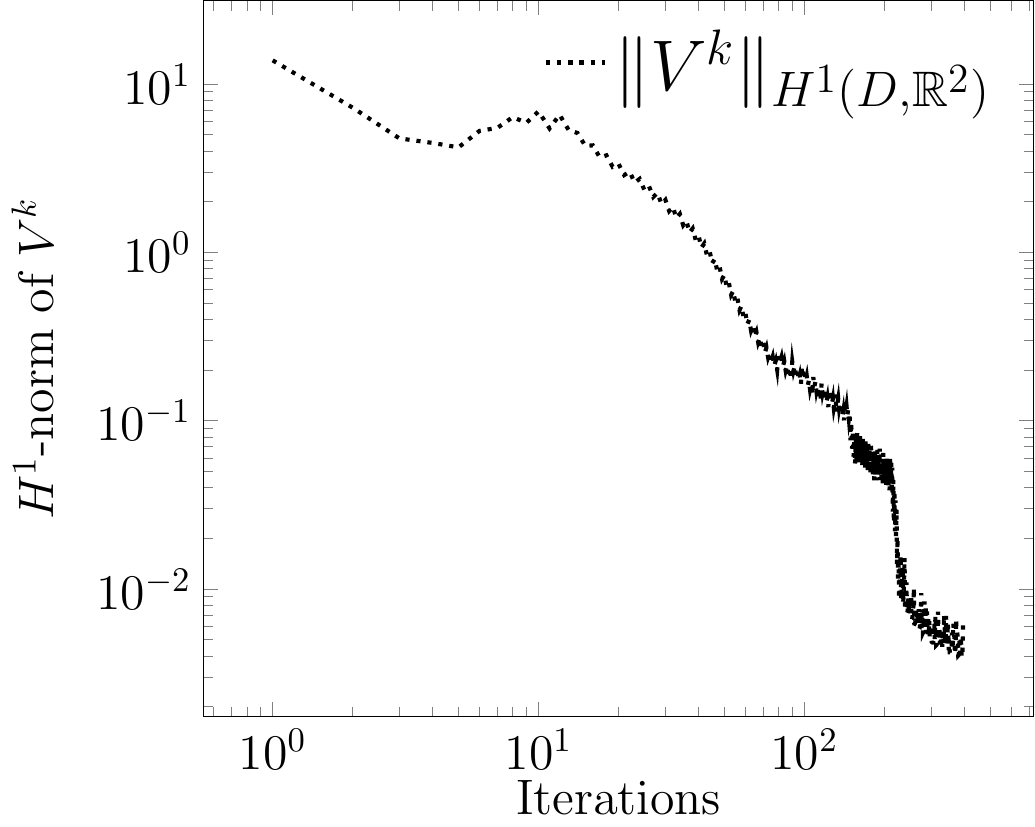}   
    \caption{Deformation vector field}
  \end{subfigure}
  \caption{Objective function and norm of the shape gradient as a function of iteration number (log/log scale).}
  \label{fig:deterministic-plots}
\end{figure}
\FloatBarrier

\subsubsection{Stochastic case: behavior of algorithm \ref{Algo:stochastic_descent}}
\label{subsection:Num_stoch}
Similar experiments to the one in subsection \ref{subsection:Num_det} are now shown. These experiments use the stochastic model formulated in subsection~\ref{Subsec:ModelProblem-Stoch} to demonstrate the performance of algorithm~\ref{Algo:stochastic_descent}. 
An example with two shapes is used again, i.e., $N=2$, and the same target shape vector $u^*$ as in subsection~\ref{subsection:Num_det} is considered. The same values for $g$ and $\nu_1=\nu_2$ are used. 

To generate samples according to the discussion at the end of subsection~\ref{Subsec:ModelProblem-Stoch}, for simplicity $\tilde{D}=D$ is used, allowing for the explicit representations of the eigenfunctions and eigenvalues in \eqref{eq:KL-expansion}. From  \cite[Example 9.37]{Lord2014}), the eigenfunctions and eigenvalues on $D$ are given by the formula
$$\tilde{\phi}_{j}^k(x):= 2\cos(j \pi x_2)\cos(k \pi x_1), \quad \tilde{\gamma}_{j}^k:=\frac{1}{4} \exp(-\pi(j^2+k^2)l^2), \quad j,k \geq 1,$$
where terms are then reordered so that the eigenvalues appear in descending order (i.e., $\phi_1 = \tilde{\phi}_{1}^1$ and $\lambda_1 = \tilde{\lambda}_{1}^1$). The correlation length $l=0.5$ and the number of summands $M=20$ is fixed. For the simplicity of presentation, each subdomain has the same eigenfunctions and eigenvalues, and only the means and random vectors are modified. More precisely, \eqref{eq:KL-expansion} has the representation

\begin{equation}
\label{eq:KL-expansion-numerics}
 \tilde{\kappa}_{i}(x,\omega) =\bar{\kappa}_{i}(x) + \sum_{k=1}^{20} \sqrt{\gamma_{k}} \phi_{k}(x) \xi_{i,k}(\omega),
\end{equation}
for every $i=0, 1, 2$. Using the same labeling convention as in the deterministic study, the values $\bar{\kappa}_0 = 1000$, $\bar{\kappa}_1 = 7.5$, and $\bar{\kappa}_2 = 5$ are used for the mean in the outer, ellipse, and tube domains, respectively. Notice that these are compatible with the choices used in the deterministic experiment. Deviations from this mean are simulated using the centered distributions $\xi_{0,k} \sim U[-50,50]$, $\xi_{1,k} \sim U[-2.5, 2.5]$, $\xi_{2,k} \sim U[-1,1]$, with $U[a,b]$ standing for the uniform distribution on the interval $[a,b] \subset\R.$ Figure~\ref{fig:random-field-realizations} shows two examples of the random fields. Since these are shown for different iterations, one also sees how a single sample in the definition of $\kappa$ is adapted to the movement of the shapes.

 \begin{figure}
 \begin{center}
  \begin{subfigure}[b]{0.5\linewidth}
  \centering
    \includegraphics[width=\textwidth]{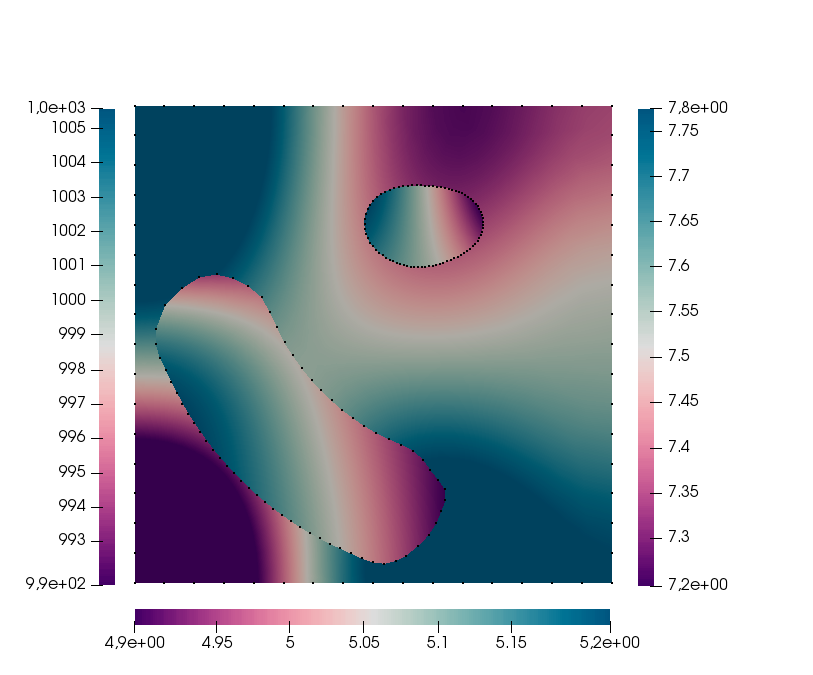} 
    \caption{Example realization of the $\kappa$ at iteration $k=100$}  
  \end{subfigure}%
  \begin{subfigure}[b]{0.5\linewidth}
    \centering
    \includegraphics[width=\textwidth]{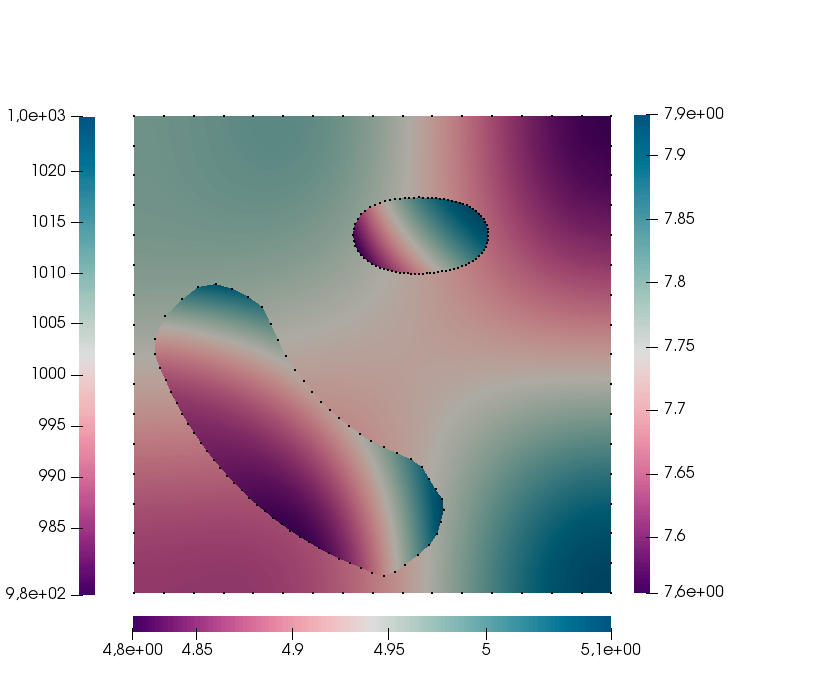}   
    \caption{Example realization of the $\kappa$ at iteration $k=300$}
  \end{subfigure}
  \end{center}
  \caption{Two examples of random field $\kappa$, with the left, right, and bottom scales corresponding to the outer domain $D_0^k$, the ellipse $D_1^k$, and the tube $D_2^k$, respectively.}
  \label{fig:random-field-realizations}
\end{figure}

The target $\bar{y}$ in the objective functional \eqref{Objective_TrackingType}  is computed by solving the \textit{deterministic} state equation \eqref{eq:PDE1d}--\eqref{eq:jumpconditionsd} on the target configuration with the mean values $\bar{\kappa}_0$, $\bar{\kappa}_1$, and $\bar{\kappa}_2$ on the target configuration $D^* = (\sqcup_{i=0}^2 D_i^*) \sqcup (\sqcup_{i=0}^2 u_i^*)$. The target is the same as in subsection~\ref{subsection:Num_det}; see figure~\ref{fig:deterministic-y-bar}.

Regarding the choice of the step-size according to \eqref{eq:Robbins-Monro-step-sizes}, experiments showed that a rule of the form $t^k = c/k$ performed poorly in practice. 
This is mostly due to the fact that the choice $c$ is limited by the fineness of the mesh; if this parameter is chosen to be too large, then the mesh deforms too drastically in the first few iterations, leading to broken meshes. However, if $c$ is chosen to be too small, the progress---although guaranteed to produce stationary points in the limit---is much too slow. To mitigate this effect, a warm start of 250 iterations using the constant step size $t^k = c=0.015$ 
is used until the shapes appear to be in the neighborhood of the optimum. Then the rule $t^k = c/(k-250)$ is used for $k=251, \dots, 400$. This produces excellent results as shown in figure~\ref{fig:random-shapes}. Even in the presence of noise, the progression of the subdomains resembles that shown in figure~\ref{fig:deterministic-shapes}. 
\begin{figure}
  \begin{subfigure}[b]{0.5\linewidth}
  \centering
    \includegraphics[width=\textwidth]{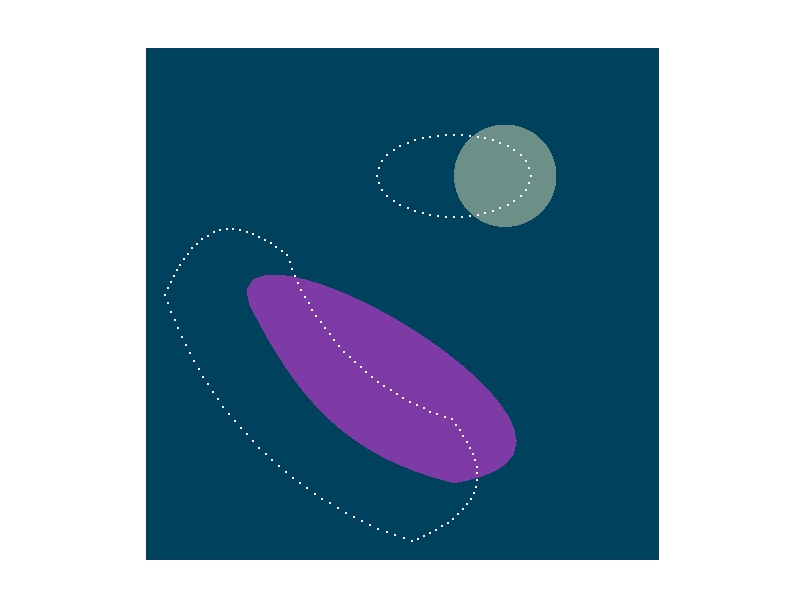} 
    \caption{Initial configuration $D^0$}  
    \label{fig:ellipse-to-circle_target-stochastic}
  \end{subfigure}%
  \begin{subfigure}[b]{0.5\linewidth}
    \centering
    \includegraphics[width=\textwidth]{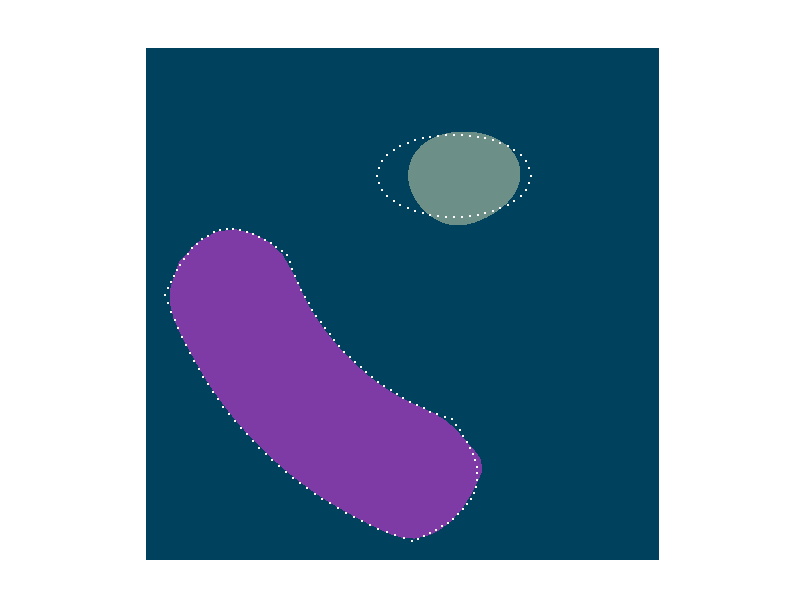}
    \caption{$D^{50}$}
  \end{subfigure}\\%
  \begin{subfigure}[b]{0.5\linewidth}
    \centering
    \includegraphics[width=\textwidth]{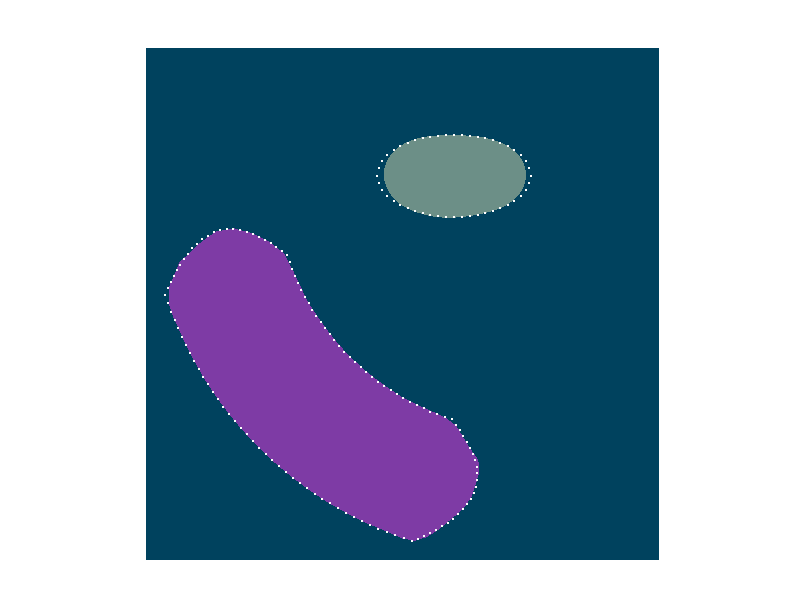}   
   \caption{$D^{200}$}
  \end{subfigure}%
    \begin{subfigure}[b]{0.5\linewidth}
    \centering
    \includegraphics[width=\textwidth]{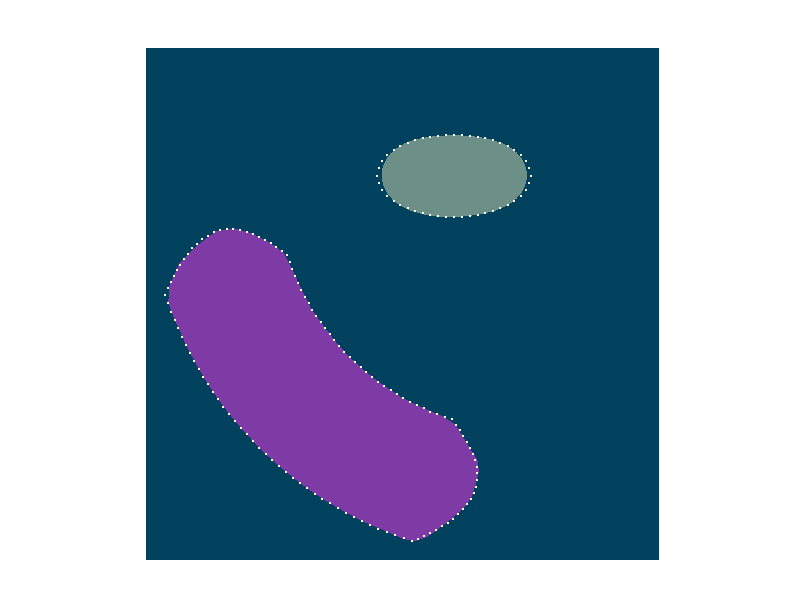} 
   \caption{$D^{400}$}
  \end{subfigure}%
  \caption{The target shapes are displayed by the dotted lines. The figures show the progression from the initial configuration of domains $D^0$ to the final configuration of domains $D^{400}$.}
  \label{fig:random-shapes}
\end{figure}

Figure~\ref{fig:stochastic-plots} provides a stochastic counterpart to figure~\ref{fig:deterministic-plots}, in which one sees the progression of the parametrized functional $J(u^k,\omega^k)$ as well as the vector field $V^k = V^k(\omega^k)$, where $\omega^k$ represents the abstract realization from the probability space in iteration $k$, which is manifested by the specific realizations of the random vectors $(\xi_{i,1}(\omega^k), \dots, \xi_{i,20}(\omega^k))$, $i=0,1,2$, used in the random fields. In contrast to the Armijo line search rule, the Robbins--Monro step-size rule does not guarantee descent in every iteration. Moreover, the information displayed in the plots can only provide estimates for the true objective $j(u^k) = \EE[J(u^k,\cdot)]$ and the average $\EE[\lVert V^k(\cdot)\rVert_{H^1(D,\R^2)}]$. Although small oscillations in the shapes were observed in the course of the algorithm, the oscillations from the plots come more from the stochastic error occuring due to $J(u^k,\omega^k) \approx \EE[J(u^k,\cdot)]$ and $\lVert V^k(\omega^k)\rVert_{H^1(D,\R^2)} \approx \EE[\lVert V^k(\omega)\rVert_{H^1(D,\R^2)}]$. The log/log scale misleadingly exaggerates these oscillations for higher iteration numbers and the Robbins--Monro step-size rule tended to dampen oscillations in the shapes for higher iterations. However, even with the oscillations, descent is seen \textit{on average} in both the parametrized objective and in the $H^1$--norm of the randomly generated deformation vector fields.

 \begin{figure}
  \begin{subfigure}[b]{0.45\linewidth}
  \centering
    \includegraphics[width=\textwidth]{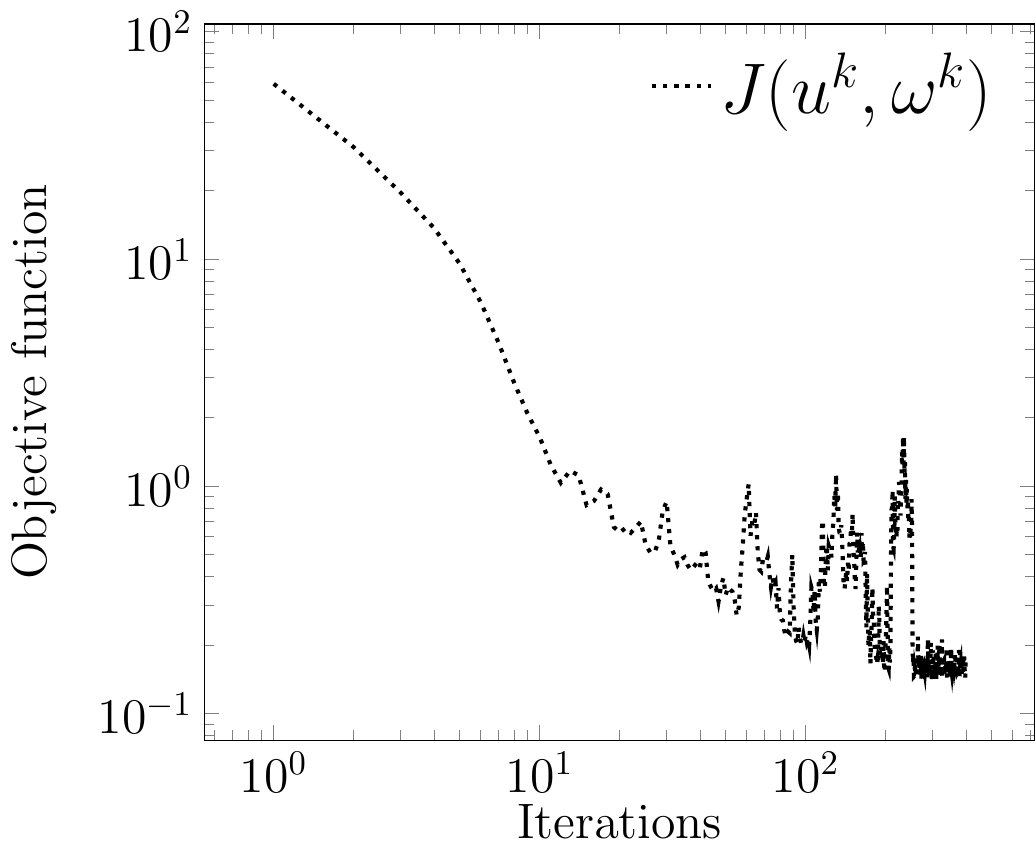} 
    \caption{Objective function decay}  
  \end{subfigure}%
  \begin{subfigure}[b]{0.45\linewidth}
    \centering
    \includegraphics[width=\textwidth]{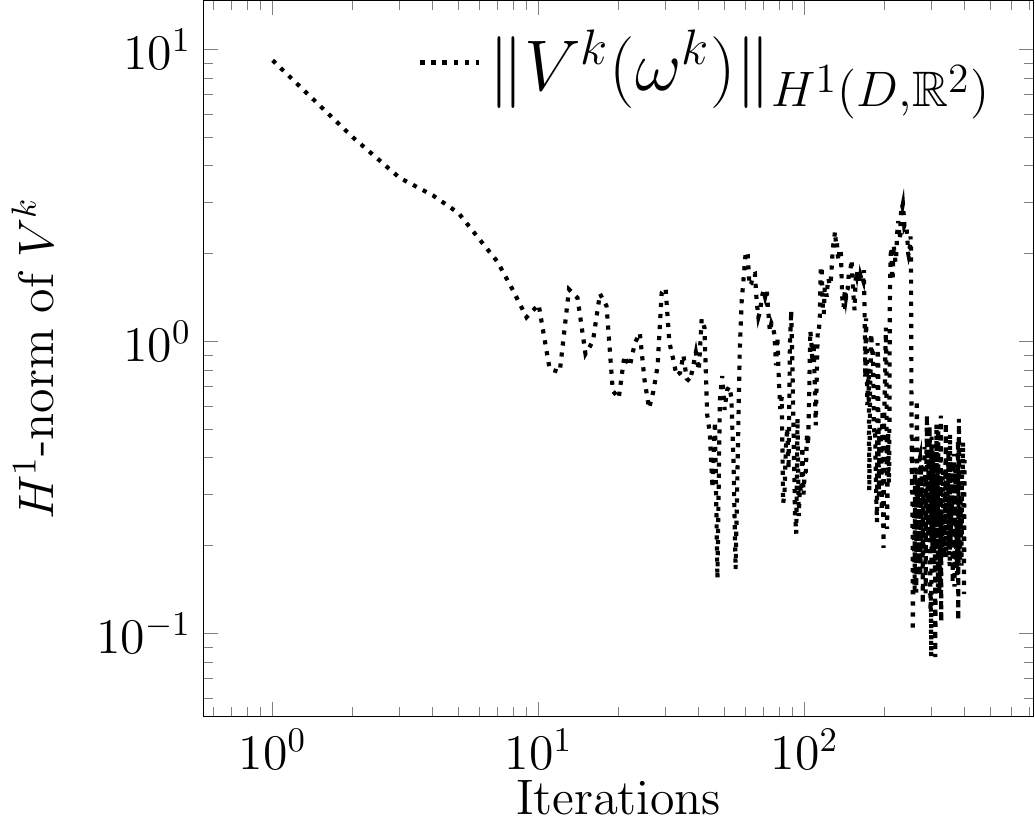}   
    \caption{Deformation vector field}
  \end{subfigure}
  \caption{Objective function and norm of the shape gradient as a function of iteration number (log/log scale).}
  \label{fig:stochastic-plots}
\end{figure}
\FloatBarrier

\subsubsection{Robustness: Deterministic vs.~stochastic model}
A final experiment justifies the use of the stochastic model if experimental parameters are uncertain. To demonstrate the concept, only a single shape is used, i.e., $N=1$. The perimeter regularization is fixed with $\nu = 5\cdot 10^{-2}$.  The expansion \eqref{eq:KL-expansion-numerics} is used for $i=0,1$ with the same eigenfunctions, eigenvalues, and choices of the correlation length $l$ and number of summands $M$. 
In each iteration $k$, on the outer domain $D_0^k$, the mean is given by $\bar{\kappa}_0 = 1000$ and distribution is chosen to be $\xi_{0,k} \sim U[-75,75]$. On the domain $D_1^k$, the mean and distribution are given by $\bar{\kappa}_1 = 7.5$ and $\xi_{1,k} \sim U[-4.5, 4.5]$.

For the generation of the target data $\bar{y}$ in the tracking-type objective functional, the target shape $u^\ast$ is chosen to be the boundary of an ellipse as illustrated by the dotted lines in figure~\ref{fig:single-shape-min-max}.
The target distribution $\bar{y}$ is computed on the target domain $D^\ast=D_0^\ast \sqcup D_1^\ast \sqcup u^\ast$  by solving the state equation \eqref{eq:PDE1d}--\eqref{eq:jumpconditionsd} using the constant values $\bar{\kappa}_0 = 1000$ over the outer domain $D^*_0$ and $\bar{\kappa}_1 = 7.5$ defined over the ellipse $D^*_{1}$. The target data can be seen in figure~\ref{fig:single-shape-y-bar}.
As in the previous experiments, algorithms are run for 400 iterations. The results of the simulation are shown in figure~\ref{fig:single-shape-min-max}, where the target shape $u^*$ is represented by dotted lines. The same initial configuration, shown in  figure~\ref{fig:single-shape-min-max_a}, is used for three separate runs of the algorithm.

In the first run, the stochastic model with the parameters described in the previous paragraph is used, and the stochastic gradient method (algorithm~\ref{Algo:stochastic_descent}) is used with the step-size rule $t^k = 0.026$ for $k=0, \dots, 200$, and $t^k = 0.026/(k-200)$ for $k=201, \dots, 400$. The configuration obtained at 400 iterations approximates the desired configuration nicely as shown in figure~\ref{fig:single-shape-min-max_b}. 

Incorrect choices for the parameters are used for the next two runs. In the disastrous case, where these parameters are incorrectly chosen at the upper or lower limits of the probability distributions, the deterministic  algorithm~\ref{Algo:multiples-shapes} does not correctly identify the desired shape $u^*$. Using the choices $\kappa_{0,\min} = 937.3$ and $\kappa_{1,\min} = 3.7$, which are chosen in such a way such that $\kappa_{i,\min} \leq \kappa_{i}(x,\omega)$ for all $(x,\omega)\in D \times \Omega$, $i=0,1$, produces the result figure~\ref{fig:single-shape-min-max_c}. Alternatively, with the choices $\kappa_{0,\max} = 1062.7$ and $\kappa_{1,\max} = 11.3$, analogously chosen so that  $\kappa_{i,\max} \geq \kappa_{i}(x,\omega)$ for all $(x,\omega)\in D\times \Omega$, $i=0,1$, results in the configuration shown in figure~\ref{fig:single-shape-min-max_d}. One clearly sees in both figure~\ref{fig:single-shape-min-max_c} and figure~\ref{fig:single-shape-min-max_d} that the correct shape is not identified, even for this very simple example.
 In summary, when parameters are subject to uncertainty, but a good model for the uncertainty is available, it is always better to use the stochastic model. The corresponding solution to the stochastic model is robust with respect to these uncertainties.

\begin{figure}
  \begin{subfigure}[b]{0.5\linewidth}
  \centering
    \includegraphics[width=\textwidth]{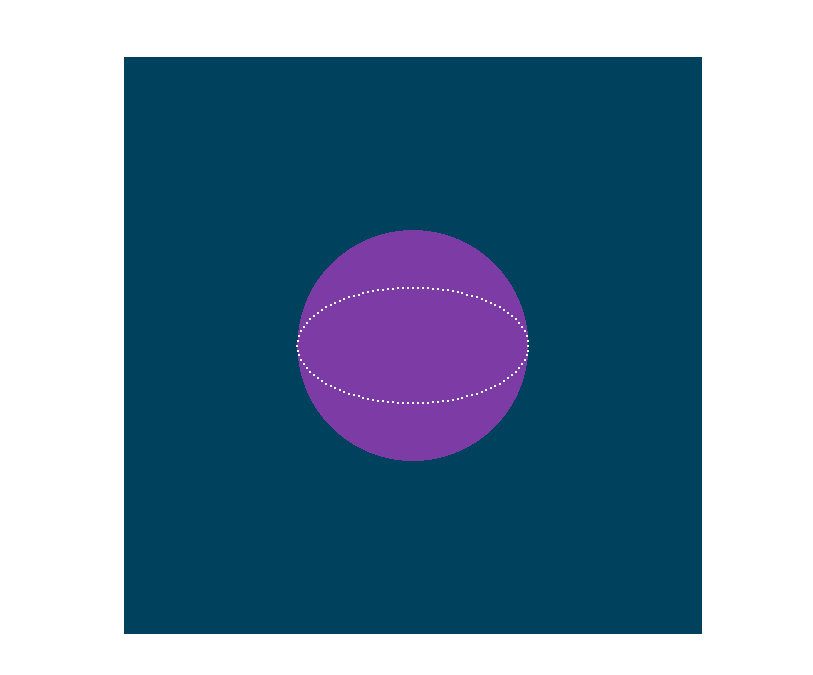} 
    \caption{Initial configuration $D^0$}  
    \label{fig:single-shape-min-max_a}
  \end{subfigure}%
  \begin{subfigure}[b]{0.5\linewidth}
    \centering
    \includegraphics[width=\textwidth]{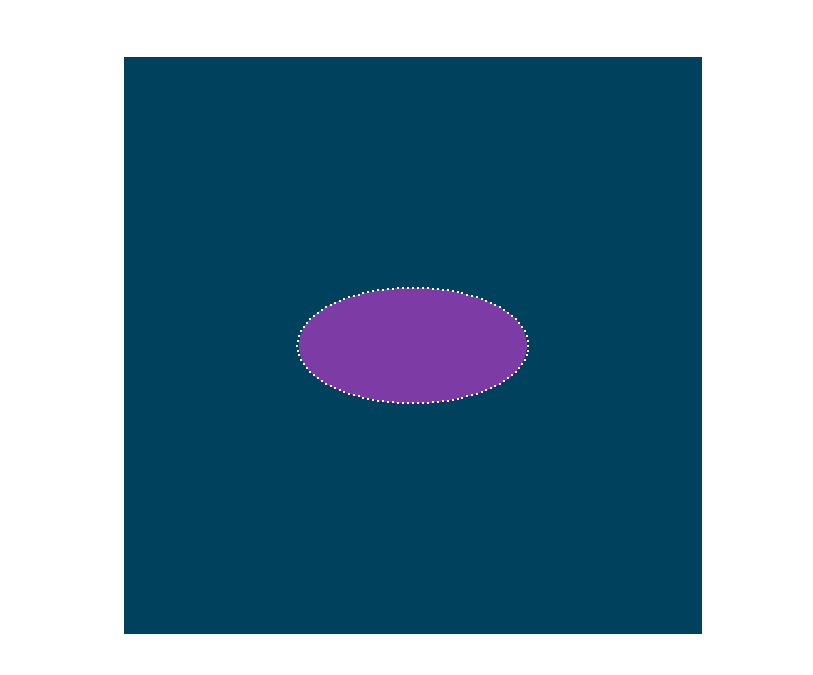}
    \caption{$D^{400}$ for stochastic model}
     \label{fig:single-shape-min-max_b}
  \end{subfigure}\\%
  \begin{subfigure}[b]{0.5\linewidth}
    \centering
    \includegraphics[width=\textwidth]{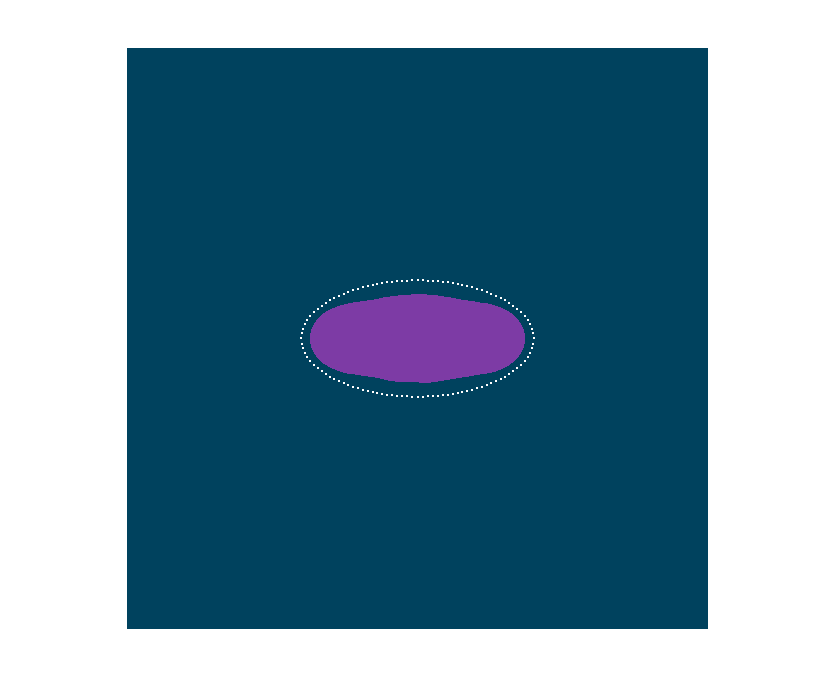}
   \caption{$D^{400}$ using the constants $\kappa_{i,\min}$, $i=0,1$}
   \label{fig:single-shape-min-max_c}
  \end{subfigure}%
    \begin{subfigure}[b]{0.5\linewidth}
    \centering
    \includegraphics[width=\textwidth]{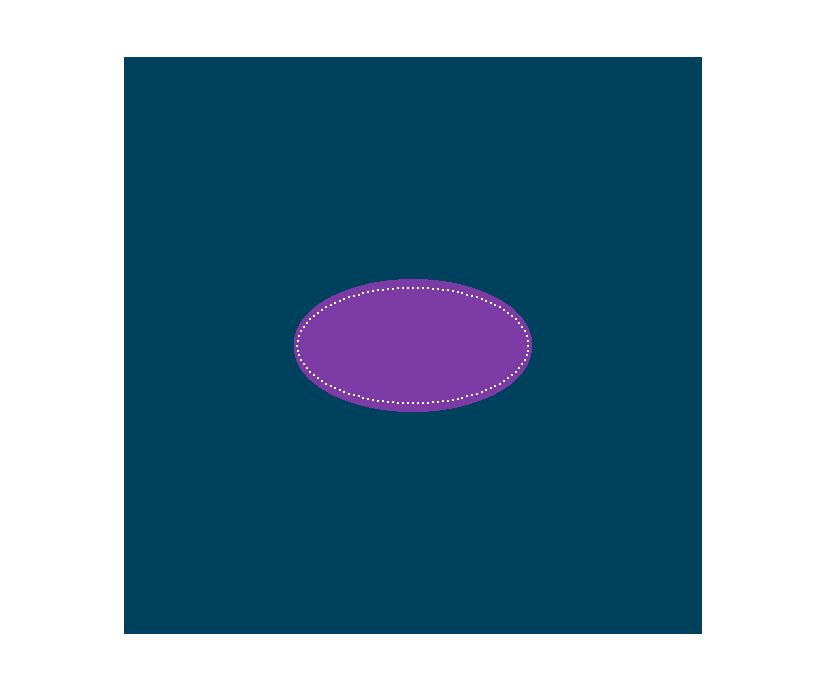}
   \caption{$D^{400}$ using the constants $\kappa_{i,\max}$, $i=0,1$}
   \label{fig:single-shape-min-max_d}
  \end{subfigure}%
  \caption{The figures show the initial configuration in (a) and the configuration computed using the stochastic model and stochastic gradient approach in (b). Using the lower bound choices produces an incorrect identification in (c); with the upper bound choices, the target shape is likewise incorrectly identified.}
  \label{fig:single-shape-min-max}
\end{figure}

  \begin{figure}
  \centering
    \includegraphics[width=0.55\textwidth]{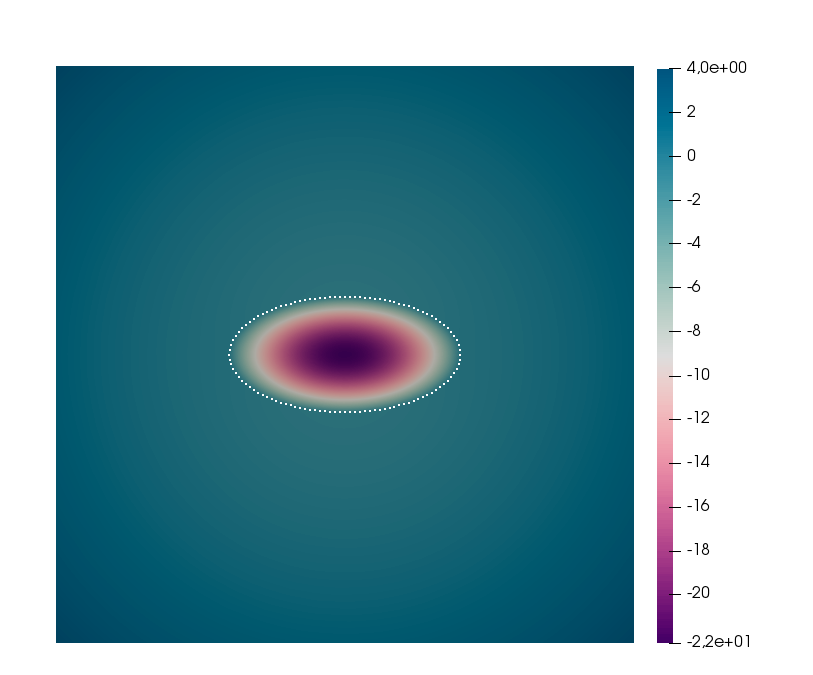}
  \caption{Values of the target data $\bar{y}$.}
  \label{fig:single-shape-y-bar}
\end{figure}

\section{Conclusion}
\label{sec:conclusion}
This chapter gives an overview how the theory of (PDE-constrained) shape optimization can be connected with the differential-geometric structure of shape space and how this theory can be adapted to handle harder problems containing multiple shapes and uncertainties. The framework presented is focused on shape spaces as Riemannian manifolds; in particular, on the space of smooth shapes and the Steklov-Poincar\'{e} metric. The Steklov-Poincar\'{e} metric allows for the usage of the shape derivative in its volume expression in optimization methods.
A novel framework developed in this chapter is a product shape shape, which allows for shape optimization over a vector of shapes. As part of this framework, new concepts including the partial and multi-shape derivatives are presented. The steepest descent method with Armijo backtracking on product shape spaces is formulated to solve a shape optimization problem over a vector of shapes.

The second area of focus in this chapter is concerned with shape optimization problems subject to uncertainty. The problem is posed as a minimization of the expectation of a random objective functional depending on uncertain parameters. Using the product shape space framework, it is no trouble to consider stochastic shape optimization problems depending on shape vectors. Corresponding definitions for the stochastic partial and multi-shape gradient are presented. These are needed  to present the stochastic gradient method on product shape spaces. It is discussed how the stochastic shape derivative in its volume expression can be used algorithmically.

The final part of the chapter is dedicated to carefully designed numerical simulations showing the performance of the algorithms. Compatible deterministic and stochastic problems are presented. A novel technique for producing stochastic samples of the Karhunen--Lo\`eve type is presented.
The stochastic model is shown in experiments to be robust if a model for the uncertainties is present.

The new framework provides a rigorous justification for computing descent vectors ``all-at-once'' on a hold-all domain. Moreover, new concepts like the partial shape derivatives and multi-shape derivatives provide tools that could be used in other applications. There are some open questions; for one, it is not clear how descent directions in general prevent shapes from intersecting as part of the optimization procedure. Mesh deformation methods like the kind used here would result in broken meshes. While the algorithms presented do not rely on remeshing, it is notable that meshes lose their integrity if initial shapes are chosen too far away from the target. These challenges will be addressed in other works.




\bibliographystyle{plain}   
\bibliography{references}

\end{document}